\documentclass[final,3p,times,number]{elsarticle}

\journal{Journal de Math\'ematiques Pures et Appliqu\'ees}

\usepackage{amsmath,amsthm,leftidx,times,pslatex,setspace}

\def\dd{{\bf d}}
\def\Dop{{\mathcal D}}
\def\Z{{\mathbb Z}}

\def\C{{\mathbb C}}
\def\Q{{\mathbb Q}}

\def\Ker{\mathrm{Ker}}

\newcommand\uno{{\mathbbm 1}}
\newcommand{\Sets}{\mathbf{\bf Sets}}

\DeclareMathOperator{\Vectf}{\bf Vect^{\it fg}}
\DeclareMathOperator{\U}{\mathcal U}

\DeclareMathOperator{\Mn}{\bf M}

\DeclareMathOperator{\At}{At}
\DeclareMathOperator{\SL}{SL}

\DeclareMathOperator{\Char}{char}
\DeclareMathOperator{\Id}{Id}

\DeclareMathOperator{\diag}{diag}
\DeclareMathOperator{\Span}{span}
\DeclareMathOperator{\Um}{\rm U}
\DeclareMathOperator{\Ind}{\mathrm Ind}

\DeclareMathOperator{\Sym}{{\rm Sym}}
\DeclareMathOperator{\Der}{Der}

\DeclareMathOperator{\DMod}{\bf DMod}

\DeclareMathOperator{\DAlg}{\bf DAlg}
\DeclareMathOperator{\Alg}{\bf Alg}

\DeclareMathOperator{\Cat}{\mathcal{C}}
\DeclareMathOperator{\Frac}{Frac}

\DeclareMathOperator{\Vect}{\bf Vect}
\DeclareMathOperator{\Hom}{Hom}
\DeclareMathOperator{\Rep}{\bf Rep}
\DeclareMathOperator{\GL}{GL}
\DeclareMathOperator{\Aut}{Aut}
\DeclareMathOperator{\Gal}{Gal}
\DeclareMathOperator{\id}{id}
\DeclareMathOperator{\End}{End}
\newcommand{\G}{\mathbf{G}}
\DeclareMathOperator{\im}{Im}
\DeclareMathOperator{\Spec}{Spec}

\newcommand{\Le}{\leqslant}
\newcommand{\Ge}{\geqslant}

\usepackage[colorlinks]{hyperref}

\theoremstyle{plain}
\newtheorem{theorem}{Theorem}[section]
\newtheorem{lemma}[theorem]{Lemma}
\newtheorem{proposition}[theorem]{Proposition}

\theoremstyle{definition}
\newtheorem{definition}[theorem]{Definition}

\newtheorem{example}[theorem]{Example}

\theoremstyle{remark}
\newtheorem{remark}[theorem]{Remark}

\usepackage{amsmath,bbm,amscd}
\usepackage[all,cmtip,matrix,arrow,curve]{xy}

\begin{document}
\begin{frontmatter}


\title{Isomonodromic differential equations and differential categories}

\author{Sergey Gorchinskiy\fnref{SG}}
\ead{gorchins@mi.ras.ru}
\address{Steklov Mathematical Institute, Gubkina str. 8, Moscow, 119991, Russia}

\author{Alexey Ovchinnikov\fnref{AO}}
\ead{aovchinnikov@qc.cuny.edu}
\address{CUNY Queens College, Department of Mathematics, 65-30 Kissena Blvd,
Queens, NY 11367, USA\\
CUNY Graduate Center, Department of Mathematics, 365 Fifth Avenue,
New York, NY 10016, USA}

\fntext[SG]{S. Gorchinskiy was supported by the grants RFBR 11-01-00145, 12-01-31506, 12-01-3302, 13-01-12420, AG Laboratory NRU HSE, RF government grant, ag.11.G34.31.0023, and by Dmitry ZiminÕs Dynasty Foundation.}

\fntext[AO]{A. Ovchinnikov was supported by the grants: NSF  CCF-0952591 and  PSC-CUNY  No.~60001-40~41.}

\begin{abstract}
We study isomonodromicity of systems of parameterized linear differential equations and related conjugacy properties of linear differential algebraic groups by means of differential categories. We prove that isomonodromicity is equivalent to isomonodromicity with respect to each parameter separately under a filtered-linearly closed assumption on the field of functions of parameters. Our result implies that one does not need to solve any non-linear differential equations to test isomonodromicity anymore. This result cannot be further strengthened by weakening the requirement on the parameters as we show by giving a counterexample. Also, we show that isomonodromicity is equivalent to conjugacy to constants of the associated parameterized differential Galois group, extending a result of P.~Cassidy and M.~Singer, which we also prove categorically. We illustrate our main results by a series of examples, using, in particular, a relation between Gauss--Manin connection and parameterized differential Galois groups.
\bigskip

\noindent{\bf R\'esum\'e}
\medskip

On \'etudie l'isomonodromie des syst\`emes d'\'equations diff\'erentielles
lin\'eaires param\'etr\'ees et les propri\'et\'es li\'ees \`a la conjugaison des
groupes alg\'ebriques diff\'erentiels lin\'eaires en utilisant les
cat\'egories diff\'erentielles. On d\'emontre que l'isomonodromie est
\'equivalente \`a l'isomonodromie relative \`a chaque param\`etre pris
s\'epar\'ement, si le corps des fonctions des param\`etres est filtr\'e
lin\'eairement clos. Ce r\'esultat implique quÕil nÕest pas n\'ecessaire
de r\'esoudre des \'equations diff\'erentielles non lin\'eaires pour tester
lÕisomonodromie. Un contre-exemple montre qu'on ne peut pas am\'eliorer
ce r\'esultat en affaiblissant la condition sur les param\`etres. On
d\'emontre, en termes de cat\'egories, que l'isomonodromie est
\'equivalente, \`a une conjugaison pr\`es, au fait que le groupe de
Galois diff\'erentiel param\'etr\'e associ\'e est constant, g\'en\'eralisant ainsi
un r\'esultat de P. Cassidy et M. Singer.  On illustre nos r\'esultats
fondamentaux par une s\'erie d'exemples utilisant, en particulier, un
lien entre la connexion de Gauss--Manin et les groupes de Galois
diff\'erentiels param\'etr\'es.
\end{abstract}

\end{frontmatter}

\section{Introduction}
A system of parameterized linear differential equations is a system of linear differential equations whose coefficients are functions of principal variables $x_1,\ldots,x_n$ and parameters $t_1,\ldots,t_d$ and derivations only with respect to $x_1,\ldots,x_n$ appear in the system.
We say that such a system is isomonodromic if it can be extended to a consistent system of linear differential equations with derivations with respect to all of  $x_1,\ldots,x_n,t_1,\ldots,t_d$.
That is, one requires that the extended system satisfies all integrability conditions with respect to the principal and parametric  variables.
In this paper, we study such isomonodromic systems via the parameterized Picard--Vessiot (PPV) theory~\cite{PhyllisMichael} and differential Tannakian categories~\cite{GGO,difftann,OvchTannakian,Moshe,Moshe2011,Besser}.

To verify isomonodromicity of a system of parameterized linear differential equations, say, with one principal variable $x$ and $d$ parameters $t_1,\ldots, t_d$ explicitly means to find
$d$ extra matrices that satisfy $\binom{d+1}{2}$ integrability conditions \cite[Definition~3.8]{PhyllisMichael}, which form a system of linear and non-linear differential equations. We improve this by showing that it is enough to find matrices that satisfy only~$d$ integrability conditions for pairs of derivations $(\partial_x,\partial_{t_i})$, which are linear differential equations, under a filtered-linearly closed assumption (Definition~\ref{defin-good}) on the field of functions of parameters (Theorems~\ref{thm:54} and~\ref{theor-expl}).  Namely, the existence of the latter matrices implies the existence of (possibly, different) matrices that satisfy all $\binom{d+1}{2}$ integrability conditions.
In other words, our result removes all non-linear differential equations from the integrability conditions that have to be tested, which now enables one to use the powerful methods of differential Galois theory to test isomonodromicity.

This result is non-trivial not only because of the method of proof (which uses differential categories~\cite{GGO} and CDG-algebras~\cite{Positselskii}) but also because it is counterintuitive. The initial explicit steps for this result restricted to $2\times 2$ systems with the parameterized differential Galois group Zariski dense in $\SL_2$ can be found in \cite[Proposition~4.4]{Dreyfus} (see also \cite[Theorem~1.3, Chapter~2]{SitSL2}).
Note that the condition on the field to be filtered-linearly closed is, indeed, necessary as is shown in Example~\ref{examp-isomon}. This example is based on iterated integrals.

A similar but more specialized question was treated in~\cite{JM,JMU} for the case of rational functions in the principal variable. Using analytic methods, it is shown that, for $d$ extra matrices of a certain special type, $d$ integrability conditions imply all $\binom{d+1}{2}$ integrability conditions for the same matrices. Additionally, it is proved in~\cite{JM,JMU} that, if the differential equation is isomonodromic (when restricted only to rational functions in the principal variable), then one can choose extra matrices of the special type discussed above (for more details, see Section~\ref{sec:anal}).

Given a system of parameterized linear differential equations, the PPV theory associates a parameterized differential Galois group, which can be represented by groups of invertible matrices whose entries are in the field of constants, that is, the field of functions of the parameters $t_1,\dots,t_d$. Moreover, these groups are linear differential algebraic groups (LDAGs), that is, groups of matrices satisfying a system of polynomial differential equations with respect of the parametric derivations
\cite{Cassidy,CassidyRep,KolDAG,OvchRecoverGroup,OvchTannakian}. Using descent for connections (Lemma~\ref{lemma-main}), we prove in Theorem~\ref{theor-main} that, under the filtered-linearly closed assumption on the field of constants, a system
of parameterized linear differential equations is isomonodromic (Definition~\ref{defin-isomon}) if and only if its Galois group is conjugate (possibly, over an extension field of the field of constants) to a group of matrices
whose entries are constant functions in the parameters. This extends the corresponding result in~\cite{PhyllisMichael}, which required the
field of constants to be differentially closed. Recall that a differential field is differentially closed if it contains solutions of all consistent systems of polynomial differential equations with coefficients in the field. Note that, even in the case of a differentially closed field of constants, our proof, based on differential Tannakian categories, is different from the one given in~\cite{PhyllisMichael}.

We construct examples showing that, in general, one really needs to take an extension of the field of constants to obtain the above conjugacy (Examples~\ref{examp-1} and~\ref{examp-2}). The construction of the examples uses an explicit description of  Galois groups for PPV extensions defined by integrals (Propositions~\ref{prop-intGalois} and~\ref{prop-PPV}), which seems to have interest in its own right. Namely, we interpret such differential Galois groups in terms of Gauss--Manin connections (Section~\ref{sec:GM}). More concretely, the examples involve the incomplete Gamma-function and the Legendre family of elliptic curves (see also \cite{Carlos2} for the computation point of view).
Note that the relation between the PPV theory and Gauss--Manin connection was also elaborated in~\cite{MichaelLAG}.

Recall that the Galois groups in the PPV theory are LDAGs. As noted above, isomonodromicity corresponds to conjugation to constants  for LDAGs. In this way, our Theorem~\ref{theor-conjug} corresponds to Theorem~\ref{thm:54} and says that if a LDAG is conjugate to groups of matrices whose entries are constants with respect to each derivation separately, then there is a common conjugation matrix, under the filtered-linearly and linearly closed assumption on the differential field. This matrix may have entries in a Picard--Vessiot extension of the base field. We construct an example showing that, in general, one  needs to take a Picard--Vessiot extension (Example~\ref{example-PV}).

As an application, we obtain a generalization of \cite[Theorem~3.14]{diffreductive}, which characterizes semisimple categories of representations of LDAGs in the case when the ground field is differentially closed and has only one derivation. In Theorem~\ref{thm:redsemisimple}, we improve this result by showing a
more general statement without these inconvenient restrictions to differentially closed fields and the case of just one derivation.

Our method is based on the new notion of differential objects in differential categories (Definition~\ref{defin-Dkobject}). We prove that there is a differential structure on an object $X$ in a differential category over a differential field $(k,D_k)$ if and only if there is a differential structure on $X$ with respect to any derivation $\partial\in D_k$, provided that $(k,D_k)$ is filtered-linearly closed (Proposition~\ref{prop-main}). This result is applied to both isomonodromic differential equations and LDAGs. We show in Example~\ref{example-new} that this result is not true over an arbitrary differential field already for the category of representations of a LDAG over $\Q(t_1,t_1)$. The example uses the Heisenberg group. Note that the application to isomonodromic differential equations requires that we work with arbitrary differential categories, not just with differential Tannakian categories or categories of representations of LDAGs (see also the discussion at the end of Section~\ref{sec:prelnot}).

In \cite{Landesman}, Landesman initiated the parameterized differential Galois theory in a more general setting based on Kolchin's axiomatic development of the differential algebraic group theory \cite{KolDAG}.
Galois theories in which Galois groups are LDAGs also appear in \cite{CharlotteArxiv,CharlotteComp,hardouin_differential_2008,HDV,CharlotteLucia,CharlotteLuciaDescent,Wibmer3,Umemura},  with the initial algorithm, when the Galois group is a subgroup in $\SL_2$, given in \cite{Dreyfus} and analytic aspects studied in~\cite{ClaudineMichael2,ClaudineMichael}. The first complete algorithm for the case of one parameter was given in~\cite{Carlos}. For the case of several parameters, further algorithms for reductive and unipotent parameterized differential Galois groups appeared in \cite{MiOvSi2,MiOvSi}, which rely on the main result of the present paper. The representation theory for LDAGs was also developed in \cite{OvchRecoverGroup,MinOvRepSL2}, and the relations with factoring partial differential equations was discussed in~\cite{PhyllisMichaelJH}. Analytic aspects of isomonodromic differential equations were studied by many authors, let us mention \cite{JM,JMU,MalgrangeSing,MalgrangeSing2}. See also a survey~\cite{Sabbah} of Bolibrukh's results on isomonodromicity and the references given there. An explicit computational approach to testing whether a system of difference equations with differential or difference parameters is isomonodromic can be found in~\cite{BOS2013,Ov2013}.

The paper is organized as follows. We start by recalling the basic definitions and properties of differential algebraic groups, differential Tannakian categories, and the PPV theory in Section~\ref{sec:prelnot}. Most of our notation is introduced in this section. The following section contains our main technical tools, Propositions~\ref{prop-main} and~\ref{prop-descent}. Section~\ref{sec:LDAG} deals with conjugating linear differential algebraic groups to constants over not necessarily differentially closed fields. The results from Section~\ref{sec:GM}, where we also establish a relation with Gauss--Manin connection, are used in order to construct non-trivial examples to Theorem~\ref{theor-main}. In Section~\ref{sec:IDE}, we show our
main results on isomonodromic systems of parameterized linear differential equations as well as illustrate them with examples justifying the necessity of the hypotheses in our main result. Also, we provide an explicit proof of the main result in Section~\ref{sec:explicit}. This proof can be used in designing algorithms. We also give an analytic interpretation of our results including the reasons that support the conclusion of the above example.

The authors thank P.~Cassidy, L.~Di Vizio, A.~Its, B.~Malgrange, A.~Minchenko, O.~Mokhov, T.~Scanlon, M.~Singer, and D.~Trushin for very helpful conversations and comments. We are also highly grateful to the referee for the helpful suggestions.

\section{Notation and preliminaries}\label{sec:prelnot}

Most of the notation and notions that we use are taken
from~\cite{GGO}. All rings are assumed to be commutative and having
a unit.
In the paper, $(k,D_k)$ stays for a differential field of zero
characteristic, that is, $k$ is a field and $D_k$ is a
finite-dimensional $k$-vector space with a Lie bracket and a
$k$-linear map of Lie rings $$D_k\to \Der(k,k)$$ that satisfies a
compatibility condition (see~\cite[Definition~3.1]{GGO}). For example, if $\partial_1,\ldots,\partial_d$ denote
commuting derivations from $k$ to itself (possibly, some of them are
zero), then $(k,D_k)$ with $$D_k:=k\cdot\partial_1\oplus\ldots
\oplus k\cdot\partial_d$$ is a differential field.

In general, the map $D_k\to \Der(k,k)$ can be non-injective, and it is possible that there is no commuting basis in $D_k$ (see~\cite[Example~3.5]{GGO}). In particular, differential fields as above include finite-dimensional Lie algebras. Moreover, many constructions become more transparent and easier to be derived if one does not choose a basis in the $k$-vector space $D_k$, for example, the definition of the de Rham complex below. This motivates our generalization of a more common notion of a differential field  (a field with commuting derivations).

Let $D_k$ in the superscript  denote taking $D_k$-constants, that is, the elements annihilated by all $\partial\in D_k$. Put $k_0:=k^{D_k}$. Put $\Omega_k:=D_k^{\vee}$. We have the de Rham complex $\Omega_k^{\bullet}$
$$
\begin{CD}
0@>>>\Omega^0_{k}@>\dd>>\Omega^1_{k}@>\dd>>\Omega^2_{k}@>\dd>>\ldots\,,
\end{CD}
$$
where $\Omega^i_k:=\wedge^i_k\Omega_k$, $i \Ge 1$, and we put  $\Omega^0_{k}:=k$ (see~\cite[Remark~3.4]{GGO}). Note that $\dd$ is $k_0$-linear, $\dd\circ\dd=0$, and $\dd$ is uniquely defined by the following Leibniz rule: $$\dd(\omega\wedge\eta)=\dd\omega\wedge\eta+(-1)^i\omega\wedge\dd\eta$$ for all $\omega\in\Omega^i_k$, $\eta\in\Omega^j_k$.

Denote the category of sets by $\Sets$. Denote the category of
$k$-vector spaces by $\Vect(k)$. Denote the category of $k$-algebras
by $\Alg(k)$. Denote the category of $D_k$-modules over $k$ by
$\DMod(k,D_k)$ (see~\cite[Definition~3.19]{GGO}). Denote the
category of $D_k$-algebras over $k$ by $\DAlg(k,D_k)$
(see~\cite[Definition~3.12]{GGO}). Denote the ring of $D_k$-polynomials in differential indeterminates $y_1,\ldots,y_n$ (see~\cite[Definition~3.12]{GGO}) by
$$
k\{y_1,\ldots,y_n\}.
$$
We say that a (possibly,
infinite-dimensional) $D_k$-module $M$ is {\it trivial} if the
multiplication map $k\otimes_{k_0}M^{D_k}\to M$ is an isomorphism
(by~\cite[Lemma~1.7]{Michael}, this map is always injective).

We say that a differential field $(k,D_k)$ is {\it linearly $D_k$-closed} if $(k,D_k)$ has no non-trivial Picard--Vessiot extensions, that is, all finite-dimensional $D_k$-modules over $k$ are trivial (see also \cite{Magid2001,ScanlonThesis,Scanlon2000}, \cite[\S3]{MagidBook}, \cite[\S0.5]{KolDAG} for the existence and use of such fields, and \cite{BMS} for analogues for difference fields). One can also iteratively apply \cite[Embedding Theorem]{Seidenberg1958} to realize such fields (if they are countable) as germs of meromorphic functions in $\dim_k(D_k)$ variables.

A functor $X:\DAlg(k,D_k)\to\Sets$ is represented by a
$D_k$-algebra $A$ if there is a functorial isomorphism
$$
X(R)\cong \Hom_{D_k}(A,R)
$$
for any $D_k$-algebra $R$. A linear $D_k$-group is a
group-valued functor on $\DAlg(k,D_k)$ that is represented by a
$D_k$-finitely generated $D_k$-Hopf algebra. Given a (pro-)linear
$D_k$-group $G$, denote the category of finite-dimensional
representations of $G$ as an affine group scheme over $k$ by
$\Rep(G)$.

Given a functor $X:\Alg(k)\to\Sets$, one traditionally denotes also
its composition with the forgetful functor $\DAlg(k,D_k)\to\Alg(k)$
by $X$. If $X$ is representable on $\Alg(k)$, then $X$ is also
representable on $\DAlg(k,D_k)$. In other words, the forgetful
functor $\DAlg(k,D_k)\to\Alg(k)$ has a left adjoint (for example,
see~\cite[\S1.2]{Gillet}), which is usually called a
prolongation. In particular, we have a representable functor
$${\mathbb A}^n:R\mapsto R^{\oplus n},$$ where $R$ is a $D_k$-algebra. Also, given a finite-dimensional $k$-vector space $V$, we have a linear $D_k$-group
$$
\GL(V):R\mapsto \Aut_R(R\otimes_k V).
$$

Given a functor $Y:\Alg(k_0)\to\Sets$, let $Y^c$ denote its
composition with the functor of $D_k$-invariants
$$
\DAlg(k,D_k)\to\Alg(k_0),\quad R\mapsto R^{D_k}.
$$
We say that functors of type $Y^c$ are {\it constant}. If $Y$ is
represented by a $k_0$-algebra $B$, then $Y^c$ is represented by
\begin{equation}\label{eq:constalg}
k\otimes_{k_0}B
\end{equation}  with the natural $D_k$-structure, where $D_k$ acts by zero on $B$. Denote the
latter $D_k$-algebra by $B^c$ and also call it constant. If $H$ is a
linear algebraic group, then $H^c$ is a constant linear $D_k$-group.
In particular, we have a representable functor
$${({\mathbb A}^n)}^c:R\mapsto {\left(R^{D_k}\right)}^{\oplus n},$$ where $R$ is a $D_k$-algebra. Also, given a finite-dimensional $k_0$-vector space $V_0$, we have the linear $D_k$-group
$$
{\GL(V_0)}^c:R\mapsto \Aut_{R^{D_k}}{\left(R^{D_k}\otimes_{k_0}V_0\right)}.
$$
It follows that there is a morphism of linear
$D_k$-groups ${\GL(V_0)}^c\to \GL(V)$, where $V:=k\otimes_{k_0}V_0$.

Note that a $D_k$-algebra $A$ is constant if and only if $A$ is
trivial as a $D_k$-module. A $D_k$-finitely generated $D_k$-algebra
$A$ is constant if and only if there is an isomorphism of
$D_k$-algebras
$$
A\cong k\{y_1,\ldots,y_n\}/I,
$$
where $I\subset k\{y_1,\ldots, y_n\}$ is a $D_k$-ideal such that, for
all $\partial\in D_k$ and $i$, $1\Le i\Le n$, the differential
polynomial $\partial y_i$ is in $I$.

For a more explicit description of constant algebras, consider a functor $$X:\DAlg(k,D_k)\to \Sets$$ represented by a reduced $D_k$-finitely generated $D_k$-algebra. Then, by the differential
Nullstellensatz (see~\cite[Theorem~IV.2.1]{Kol}), $X$ is constant if and
only if there is a Kolchin closed embedding $X\subset {\mathbb A}^n$
over $(k,D_k)$ such that, for a $D_k$-closed field $\U$ over $k$
(equivalently, for any $\U$ as above), we have
$$
X(\U)\subset \U_0^n,\quad \U_0:=\U^{D_k},
$$
that is, all points in $X\subset {\mathbb A}^n$ have constant coordinates.

Given a $D_k$-object $X$ over $k$ (e.g., a $D_k$-module, a $D_k$-algebra, a linear $D_k$-group) and a $D_k$-field $l$ over $k$, let $X_l$ denote the $D_l$-object over $l$ obtained by the extension of scalars from $(k,D_k)$ to $(l,D_l)$, where $D_l:=l\otimes_k D_k$.

One finds the definition of a parameterized differential field in \cite[\S3.3]{GGO}. Recall that, for a parameterized differential field $(K,D_K)$ over $(k,D_k)$, one has a $K$-linear map $$D_K\to K\otimes_k D_k,$$ called a structure map. This defines a differential field $\left(K,D_{K/k}\right)$, where $D_{K/k}$ is the kernel of the structure map. Also, one has $K^{D_{K/k}}=k$. For example, if $$k=\C(t),\quad K=\C(t,x),\quad D_k=k\cdot \partial_t,\quad \text{and}\quad D_K=K\cdot\partial_x\oplus K\cdot\partial_{t},$$ then $(K,D_K)$ is a parameterized differential field over $(k,D_k)$ with $D_{K/k}=K\cdot\partial_x$.

Given a finite-dimensional $D_{K/k}$-module $N$ over $K$, one has the notion of a parameterized Picard--Vessiot (PPV) extension $L$ for $N$, where $L$ is a $D_K$-field over $K$. This was first defined in~\cite{PhyllisMichael} (see also~\cite[Definition~3.27]{GGO} for the present approach to parameterized differential fields). If $k$ is $D_k$-closed, then a PPV extension exists for any $N$ as above (see~\cite[Theorem~3.5(1)]{PhyllisMichael}). Given a PPV extension $L$, one shows that the group-valued functor
$$
\Gal^{D_K}(L/K):\DAlg(k,D_k)\to \Sets,\quad R\mapsto
\Aut^{D_K}(R\otimes_k A/R\otimes_k K)
$$
is a linear $D_k$-group (see~\cite[Lemma~8.2]{GGO}), which is called the parameterized differential Galois group of $L$ over $K$, where $A$ is the PPV ring associated to $L$ (see~\cite[Definition~3.28]{GGO}).

The main notion defined in~\cite{GGO} is that of a differential category. A $D_k$-category $\Cat$ over $k$ is an abelian $k$-linear tensor category together with exact $k$-linear endofunctors $\At^1_{\Cat}$ and $\At^2_{\Cat}$, called Atiyah functors, that satisfy a list of axioms (see~\cite[\S\S4.2,4.3]{GGO}). In particular, for any object $X$ in $\Cat$, there is a functorial exact sequence
\begin{equation}\label{eq:Atsequence}
\begin{CD}
0@>>>\Omega_k\otimes_k X@>i_X>> \At^1_{\Cat}(X)@>\pi_X>> X@>>> 0,
\end{CD}
\end{equation}
where, as above, $\Omega_k=D^\vee_k$, and a functorial embedding
$$
\At^2_{\Cat}(X)\subset \At^1_{\Cat}\left(\At^1_{\Cat}(X)\right).
$$
We have the equality
\begin{equation}\label{eq:441}
\Sym^2_k\Omega_k\otimes_k X=\At^2_{\Cat}(X)\cap\left(\Omega_k\otimes_k\Omega_k\otimes_k X\right)\subset \At^1_{\Cat}{\left(\At^1_{\Cat}(X)\right)},
\end{equation}
and both compositions
\begin{equation}\label{eq:446}
\begin{CD}
\At^2_{\Cat}(X)@>>> \At^1_{\Cat}\left(\At^1_{\Cat}(X)\right)@>\At_{\Cat}^1(\pi_X)>>\At^1_{\Cat}(X),
\end{CD}
\end{equation}
\begin{equation}\label{eq:452}
\begin{CD}
\At^2_{\Cat}(X)@>>> \At^1_{\Cat}\left(\At^1_{\Cat}(X)\right)@>\pi_{\At^1(X)}>>\At^1_{\Cat}(X)
\end{CD}
\end{equation}
are surjective (see~\cite[Lemma~4.14, Proposition~4.18]{GGO}).

For example, $\Vect(k)$ has a canonical $D_k$-structure given by the usual Atiyah extension (see~\cite[Example~4.7]{GGO}). It can be defined either in terms of jet rings (see~\cite[\S3.6, Example~4.7]{GGO}), or in terms of differential operators as follows. Given a finite-dimensional $k$-vector space $V$, the $k$-vector space $\At^1(V)$ consists of first order $D_k$-differential operators from $V^{\vee}$ to $k$. If $\dim_k D_k =1$ and $\Omega_k=k\cdot\omega$, then we have:
$$
\At^1(V)=V\otimes_k(k\oplus k\cdot\omega),
$$
and the $k$-linear structure on it is defined as follows:
$$
a\cdot(u\otimes 1+ v\otimes\omega) =au\otimes 1+u\otimes\dd a + av\otimes \omega,\quad a\in k,\ \ u,v\in V.
$$
In this case, the morphisms in the exact sequence~\eqref{eq:Atsequence} are given by
$$
i_V(v\otimes\omega) = v\otimes\omega\quad\text{and}\quad \pi_V(u\otimes 1+v\otimes\omega)= u,\quad u,v\in V,
$$
respectively.
The above approach is dual to the one in~\cite[Definition~1]{OvchTannakian} and~\cite[pp.~1199--1200]{diffreductive}.
The $D_k$-structure on $\Vect(k)$ induces a canonical $D_k$-structure on $\Rep(G)$ for a (pro-)linear $D_k$-group $G$ (see~\cite[Example~4.8]{GGO}).

Another important example is the category $$\DMod(K,D_{K/k})$$ of $D_{K/k}$-modules over $K$, where $(K,D_K)$ is a parameterized differential field over $(k,D_k)$. The category $\DMod(K,D_{K/k})$ is $k$-linear and has a canonical $D_k$-structure (see~\cite[Theorem~5.1]{GGO}). In \cite[\S4.4]{GGO}, the authors investigate differential Tannakian categories. Any neutral $D_k$-Tannakian category with a $D_k$-fiber functor is equivalent to $\Rep(G)$ with the forgetful functor, where $G$ is a (pro-)linear $D_k$-group. Note that the category $\DMod(K,D_{K/k})$ is not necessarily a $D_k$-Tannakian category (even without the requirement of being neutral), because the category $\Vect(K,D_K)$ does not necessarily have a structure of a $D_k$-category.

\section{$D_k$-structure on objects in $D_k$-categories}\label{sec:categorical}

In this section, we define a $D_k$-structure on objects in abstract $D_k$-categories. This notion and its main property given in Proposition~\ref{prop-main} are used in Sections~{sec:LDAG} and~\ref{sec:IDE} for applications to linear $D_k$-groups and isomonodromic parameterized linear differential equations, respectively. As we will further see in Example~\ref{example-new}, the filtered-linearly closed assumption of the proposition cannot be removed.

Let $\Cat$ be a $D_k$-category over $k$, $X$ be an object in $\Cat$.

\begin{definition}\label{defin-Dkobject}
A {\it $D_k$-connection} on $X$ is a section
$$
s_X:X\to\At^1_{\Cat}(X)
$$
of $\pi_X$. A $D_k$-connection $s_X$ is a {\it $D_k$-structure} if the
image of the composition
$$
\begin{CD}
X@>s_X>> \At^1_{\Cat}(X)@>\At^1_{\Cat}(s_X)>> \At^1_{\Cat}\left(\At^1_{\Cat}(X)\right)
\end{CD}
$$
is contained in $\At^2_{\Cat}(X)$.
\end{definition}

\begin{example}\label{ex:32}
To give a $D_k$-connection on a $k$-vector space $V$ as an object in $\Cat=\Vect(k)$ is the same as to give a usual connection on $V$, that is, a map
$$
\nabla_V:V\to \Omega_k\otimes_k V
$$
that satisfies the Leibniz rule. A $k$-vector space together with a $D_k$-structure is the same as a $D_k$-module (see~\cite[Proposition~3.42]{GGO}).
\end{example}

Let us give an equivalent condition for the existence of a $D_k$-connection. For any $\partial\in D_k$, the $D_k$-category $\Cat$ has a canonical
structure of a $\partial$-category by \cite[Proposition~4.12(i)]{GGO}.
Explicitly, a calculation shows that, for each object $X$ in $\Cat$, we have
$$
\At^1_{\Cat,\partial}(X)=\At^1_{\Cat}(X)/U,\quad\text{where}\quad U := \Ker {\left(\partial\otimes\id_X :\Omega_k\otimes_kX \to X\right)}
$$
and $\At^1_{\Cat,\partial}$ is the Atiyah functor that corresponds to the $\partial$-category structure on $\Cat$. Denote the quotient morphism by
$$
\alpha_{\partial}:\At^1_{\Cat}(X)\to \At^1_{\Cat,\partial}(X).
$$
 Since $U$ is contained in $\Omega_k\otimes_k X$, the morphism $\pi_X:\At^1_{\Cat}(X)\to X$ factors through $\alpha_{\partial}$. That is, we obtain a morphism
$$
\pi_{X,\partial}:\At^1_{\Cat,\partial}(X)\to X
$$
such that $\pi_{X,\partial}\circ\alpha_{\partial}=\pi_X$. By definition, a $\partial$-connection on $X$ is a section of $\pi_{X,\partial}$.

\begin{proposition}\label{prop-connequiv}
There is a $D_k$-connection on $X$ if and only if there is a basis $\partial_1,\ldots,\partial_d$ in $D_k$ over $k$ such that for any $i$, there is a $\partial_i$-connection on $X$.
\end{proposition}
\begin{proof}
The existence of a $D_k$-connection on $X$ implies the existence of a $\partial$-connection on $X$ for any $\partial\in D_k$ by the explicit construction of $\At^1_{\Cat,\partial}$ given above.

Now let us show the reverse implication. The morphisms $\alpha_{\partial_i}$, $1\leqslant i\leqslant d$, defined above give a morphism
$$
\alpha:\At^1_{\Cat}(X)\to Z
$$
such that $\pi\circ\alpha=\pi_X$, where
$$
Z:=\At^1_{\Cat,\partial_1}(X)\times_X\ldots\times_X \At^1_{\Cat,\partial_d}(X)\stackrel{\pi}\longrightarrow X.
$$
is the fibred product in $\Cat$. Thus, we have the following commutative diagram:
$$
\begin{CD}
0@>>>\Ker(\pi_X)@>>>\At^1_{\Cat}(X)@>\pi_X>>X@>>>0\\
@.@VVV@V\alpha VV@V\id_X VV\\
0@>>>\Ker(\pi)@>>>Z@>\pi>>X@>>>0\\
\end{CD}
$$
Since $\partial_1,\ldots,\partial_d$ is a basis of $D_k$ over $k$, the map
$$
\bigoplus_{i=1}^d\,\partial_i:\Omega_k\to k^{\oplus d}
$$
is an isomorphism. It follows that the restriction of $\alpha$ to $\Ker(\pi_X)=\Omega_k\otimes_k X$ is an isomorphism
$$
\begin{CD}
\Ker(\pi_X)@>{\oplus_i (\partial_i\otimes\id_X)}>>
\Ker(\pi)=\bigoplus\limits_{i=1}^d\Ker{\left(\pi_{X,\partial_i}\right)}=
\bigoplus\limits_{i=1}^d X.
\end{CD}
$$
Thus, $\alpha$ itself is an isomorphism.  Hence, given sections~$s_i$ of the morphisms $\pi_i$ for all $i$, $1\leqslant i\leqslant d$, we obtain a section $s_X$ of $\pi_X$.
\end{proof}

In what follows, we address the question whether the existence of a $D_k$-connection on $X$ implies the existence of a $D_k$-structure on $X$. It will be convenient to use the following notion first introduced in~\cite{Positselskii}. Recall that, for a graded associative algebra $$A^{\bullet}=\bigoplus_i A^i,$$ the commutator is defined by the formula
$$
[a,b]:=a\cdot b-(-1)^{\deg(a)\deg(b)}b\cdot a
$$
for homogenous elements $a,b\in A^{\bullet}$.

\begin{definition}\label{defin-CDGA}
A {\it CDG-structure} on a graded associative algebra
$A^{\bullet}$ over a field $F$ is a pair $(\dd,h)$, where
$$
\dd: A^i\to A^{i+1}
$$
is a collection of $F$-linear maps that satisfy the graded Leibniz rule
$$
\dd(a\cdot b)=\dd(a)\cdot b+(-1)^{\deg(a)}a\cdot \dd(b)
$$
for all homogenous elements $a,b\in A^{\bullet}$, and $h\in A^2$ is such that
$$
(\dd\circ \dd)(\cdot)=[h,\cdot\:],\quad \dd(h)=0.
$$
\end{definition}

Given a CDG-structure $(\dd,h)$ on $A^{\bullet}$ and an element $a\in A^1$, we obtain a new CDG-structure with
\begin{equation}\label{eq:583}
\dd'=\dd+[a, \cdot\:],\quad h'=h+\dd(a)+a^2.
\end{equation}
By definition, the CDG-structures $(\dd,h)$ and $(\dd',h')$ are {\it equivalent}.

\begin{example}\label{examp-CDGA}
\hspace{0cm}
\begin{enumerate}
\item\label{en:523}
The pair $(\dd,0)$ defines a CDG-structure on the graded associative algebra $\Omega^{\bullet}_k$ over $k_0$, where $\dd$ denotes the differential in the de Rham complex.
\item\label{en:526}
Let $V$ be a $k$-vector space, $\nabla_V$ be a $D_k$-connection on $V$. We obtain maps
$$
\nabla_V:\Omega^i_k\otimes_k V\to \Omega^{i+1}_k\otimes_k V,\quad
\nabla_V(\omega\otimes v):=\dd\omega\otimes v+(-1)^i\omega\wedge\nabla_V(v)
$$
and a CDG-structure $(\dd,h)$ on the graded associative algebra $\Omega_k^{\bullet}\otimes_k \End_k(V)$ over $k_0$ with
$$
\dd(a):=\left(\id_{\Omega_k}\wedge\,a\right)\circ\nabla_V-(-1)^i\nabla_V\circ a,\quad a\in \Omega_k^i\otimes_k\End_k(V)=\Hom_k(V,\Omega^i_k\otimes_k V),
$$
$$
h:=\nabla_V\circ\nabla_V\in \Omega_k^2\otimes_k\End_k(V)=\Hom_k(V,\Omega^2_k\otimes_k V).
$$
The condition $\dd(h)=0$ is classically called the second Bianchi identity. Note that $h$ vanishes if and only if the connection $\nabla_V$ is a $D_k$-structure on $V$. The natural embedding
$$
\Omega^{\bullet}_k\subset \Omega^{\bullet}_k\otimes_k\End_k(V)
$$
given by $\id_V\in \End_k(V)$ commutes with $\dd$. Thus, the notation $\dd$ in the CDG-structure on $\Omega_k^{\bullet}\otimes_k \End_k(V)$ does not lead to a contradiction.
\end{enumerate}
\end{example}

There is a notion of a morphism between differential fields $(k,D_k)\to (K,D_K)$, which generalizes $D_k$-fields over $k$ (see~\cite[Definition~3.6]{GGO}) In particular, we have a canonical $k$-linear map $\Omega_k\to \Omega_K$. Given such a morphism, one defines differential functors from $D_k$-categories over $k$ to $D_K$-categories over $K$ (see~\cite[Definition~4.9]{GGO}). For example, if $\Cat$ is a Tannakian category, then there is a faithful differential functor $\Cat\to\Vect(K)$ for a $D_k$-field $K$ over $k$.
The following result generalizes Example~\ref{examp-CDGA}\eqref{en:526}.

\begin{lemma}\label{lemma-connCDGA}
Let $s_X$ be a $D_k$-connection on $X$. Suppose that there is a morphism of differential fields $(k,D_k)\to (K,D_K)$ such that the map $\Omega_k\to \Omega_K$ is injective and a faithful differential functor $F:\Cat\to \Vect(K)$. Then the following is true:
\begin{enumerate}
\item\label{en:570}
$s_X$ defines a CDG-structure $(\dd,h)$ on the graded associative algebra $\Omega^{\bullet}_k\otimes_k\End_{\Cat}(X)$ over $k_0$;
\item\label{en:573}
$h$ vanishes if and only if $s_X$ is a $D_k$-structure;
\item\label{en:575}
given a CDG-structure $(\dd',h')$ on $\Omega^{\bullet}_k\otimes_k\End_{\Cat}(X)$, there is a $D_k$-connection $s'_X$ on $X$ such that $(\dd',h')$ corresponds to $s'_X$ if and only if $(\dd',h')$ is equivalent to $(\dd,h)$.

\end{enumerate}
\end{lemma}
\begin{proof}
First let us show~\eqref{en:570}. The section $s_X$ defines a map
$$
\nabla:\End_{\Cat}(X)\to\Omega_k\otimes_k\End_{\Cat}(X),\quad \nabla(a):=s_X\circ a-\At^1_{\Cat}(a)\circ s_X.
$$
In other words, $\nabla(a)$ measures non-commutativity of the diagram
$$
\begin{CD}
X@>s_X>>\At^1_{\Cat}(X)\\
@Va VV @V\At^1(a) VV\\
X@>s_X>>\At^1_{\Cat}(X).
\end{CD}
$$
One checks that $\nabla$ is a $D_k$-connection on the $k$-algebra $\End_{\Cat}(X)$. By the (graded) Leibniz rule, this extends uniquely to a collection of $k_0$-linear maps
$$
\dd:\Omega^i_k\otimes_k\End_{\Cat}(X)\to \Omega^{i+1}_k\otimes_k\End_{\Cat}(X).
$$
Next, let us define $h\in \Omega^2_k\otimes_k\End_{\Cat}(X)$. Put
$$
Y:=\Ker\big(\At^1_{\Cat}(\pi_X)-\pi_{\At^1(X)}:
\At^1_{\Cat}\left(\At^1_{\Cat}(X)\right)\to\At^1_{\Cat}(X)\big).
$$
We claim that the image of the composition
$$
\At^1_{\Cat}(s_X)\circ s_X:X\to\At^1_{\Cat}\left(\At^1_{\Cat}(X)\right)
$$
is contained in $Y$. To prove this, recall that $\pi_X\circ s_X=\id_X$. Since $\At^1_{\Cat}$ is a functor, we have
$$
\At^1_{\Cat}(\pi_X)\circ \At^1_{\Cat}(s_X)=\id_{\At^1(X)},
$$
whence
\begin{equation}\label{eq:694}
\At^1_{\Cat}(\pi_X)\circ \At^1_{\Cat}(s_X)\circ s_X=s_X.
\end{equation}
Since the morphism $\At^1_{\Cat}(X)\stackrel{\pi_X}\longrightarrow X$ is functorial in $X$, the following diagram commutes:
$$
\begin{CD}
\At^1_{\Cat}(X)@>\At^1(s_X)>>\At^1_{\Cat}\left(\At^1_{\Cat}(X)\right)\\
@V\pi_X VV @V\pi_{\At^1(X)}VV\\
X@>s_X >>\At^1_{\Cat}(X).
\end{CD}
$$
Hence, we have
\begin{equation}\label{eq:707}
\pi_{\At^1(X)}\circ \At^1_{\Cat}(s_X)\circ s_X=s_X\circ\pi_X\circ s_X=s_X.
\end{equation}
Combining~\eqref{eq:694} and~\eqref{eq:707}, we obtain the following equality of morphisms from $X$ to $\At^1_{\Cat}\left(\At^1_{\Cat}(X)\right)$:
$$
\At^1_{\Cat}(\pi_X)\circ \At^1_{\Cat}(s_X)\circ s_X=\pi_{\At^1(X)}\circ \At^1_{\Cat}(s_X)\circ s_X.
$$
Thus, the image of $\At^1_{\Cat}(s_X)\circ s_X$ is contained in $Y$.

Since $F:\Cat\to \Vect(K)$ is faithful, we have that $\At^2_{\Cat}(X)\subset Y$ (see~\cite[Remark~4.21(iii)]{GGO}). By the construction of $Y$, we have the following exact sequence
$$
\begin{CD}
0@>>>\Omega_k\otimes_k\Omega_k\otimes_k X@>>>Y@>\At^1(\pi_X)>>\At^1_{\Cat}(X)@>>>0.
\end{CD}
$$
By~\eqref{eq:441} and~\eqref{eq:446} (see \S\ref{sec:prelnot}), we obtain an isomorphism
$$
\Omega^2_k\otimes_k X\stackrel{\sim}\longrightarrow Y/\At^2_{\Cat}(X).
$$
Put
$$
h\in\Omega^2_k\otimes_k\End_{\Cat}(X)=\Hom_{\Cat}(X,\Omega^2_k\otimes_k X)
$$
to be the composition
$$
\begin{CD}
X@>\At^1(s_X)\circ s_X>> Y@>>> Y/\At^2_{\Cat}(X)@>\sim>>\Omega^2_k\otimes_k X.
\end{CD}
$$
It remains to show the identities
\begin{equation}\label{eq:1709}
\dd\circ \dd=[h,\cdot\:],\quad \dd(h)=0.
\end{equation}

One can show that, if $\Cat$ is the category of vector spaces over a differential field, then $\dd$ and $h$ constructed as above coincide with those defined in Example~\ref{examp-CDGA}\eqref{en:526}.  Further, the constructions of $\dd$ and $h$ commute with differential functors. More explicitly, consider the differential functor $F:\Cat\to\Vect(K)$. The morphism of differential fields $(k,D_k)\to (K,D_K)$ defines a homomorphism of graded algebras
$\Omega^{\bullet}_k\to\Omega_K^{\bullet}$, which commutes with the de Rham differential $\dd$ (see~\cite[Definition~3.6]{GGO}). The functor $F$ induces a homomorphism of graded algebras
$$
\alpha:\Omega^{\bullet}_k\otimes_k\End_{\Cat}(X)\to \Omega^{\bullet}_K\otimes_K\End_K\left(F(X)\right).
$$
The connection $s_X$ on $X$ defines a $D_K$-connection on the $K$-vector space $F(X)$ such that $\alpha$ commutes with $\dd$ and preserves $h$. Since $F$ is faithful and the map $\Omega_k\to\Omega_K$ is injective, $\alpha$ is injective. Thus, we obtain~\eqref{eq:1709} by Example~\ref{examp-CDGA}\eqref{en:526} applied to $K$-vector spaces. This finishes the proof of~\eqref{en:570}.

Further,~\eqref{en:573} follows from the construction of $h$. To prove~\eqref{en:575}, note that any other $D_k$-connection on $X$ is given~by
\begin{equation}\label{eq:allconnections}
s'_X=s_X+a,
\end{equation}
where
$$
a\in \Omega^1_k\otimes_k\End_{\Cat}(X)
$$
is an arbitrary element. We need to show that the corresponding CDG-structure $(\dd',h')$ on $\Omega^{\bullet}_k\otimes_k\End_{\Cat}(X)$  satisfies~\eqref{eq:583}.
As above, by the injectivity of the algebra homomorphism $\alpha$, it is enough to consider the case $\Cat=\Vect(K)$, in which the required follows from Example~\ref{examp-CDGA}\eqref{en:526}.
\end{proof}

It follows from Lemma~\ref{lemma-connCDGA} that if $\dim_k(D_k)=1$ and $\Cat$ satisfies the condition from Lemma~\ref{lemma-connCDGA}, then any $D_k$-connection $s_X$ on an object $X$ in $\Cat$ is a $D_k$-structure on $X$.

One can give a different definition of a $D_k$-category so that Lemma~\ref{lemma-connCDGA} holds for any $D_k$-category in this new sense. Namely, one should require the compatibility condition from~\cite[Remark~4.21(i)]{GGO} and also the pentagon condition for $\Psi$ in notation from there. The latter condition involves consideration of the third jet-ring $P^3_k$.

\begin{definition}\label{defin-good}
We say that a differential field $(k,D_k)$ is {\it filtered-linearly closed} if there is a sequence of $k$-vector subspaces closed under the Lie bracket
$$
0=D_0\subset D_{1}\subset\ldots\subset D_{d-1}\subset D_{d}=D_k
$$
such that for any $i$, $0\leqslant i\leqslant d-1$, we have
$$
\dim_k\left(D_{i+1}/D_{i}\right)=1
$$
and $k$ is linearly $D_i$-closed.
\end{definition}

Note that, in Definition~\ref{defin-good}, we do not require that $k$ be linearly $D_k$-closed, that is, a filtered-linearly closed field is not necessarily linearly closed.

\begin{example}\label{examp-good}
\hspace{0cm}
\begin{enumerate}
\item
If $\dim_k(D_k)=1$, then $(k,D_k)$ is filtered-linearly closed.
\item
If $k$ is $D_k$-closed, then $(k,D_k)$ is filtered-linearly closed. Indeed, since $k$ is
$D_k$-closed, the natural map $D_k\to \Der(k,k)$ is injective. By
\cite[p.~12, Proposition~6]{KolDAG}, there is a commuting basis
$\partial_1,\ldots,\partial_d$ in $D_k$ over $k$ and we put
$$
D_i:=\Span_k\langle\partial_1,\ldots,\partial_i\rangle.
$$
Again, since $k$ is $D_k$-closed, we see that $k$ is linearly $D_i$-closed.
\end{enumerate}
\end{example}

\begin{lemma}\label{lemma-CDGA}
Let $A$ be a finite-dimensional associative algebra over $k$. Suppose that $(k,D_k)$ is filtered-linearly closed. Then any CDG-structure on $\Omega_k^{\bullet}\otimes_k A$ is equivalent to a CDG-structure with $h=0$.
\end{lemma}
\begin{proof}
We use induction on $d:=\dim_k(D_k)$. The case $d=1$ is automatic. Let us make the inductive step from $d-1$ to $d$. Consider the differential fields $(k,D_{i})$ from Definition~\ref{defin-good} and put $\Omega^{\bullet}_{i}$ to be the corresponding de Rham complexes. Also, put
$$
\Omega:=\Ker\left(\Omega_d\to\Omega_{d-1}\right).
$$
Since $D_{d-1}$ is a Lie subring in $D_d$, we have a morphism of differential fields $(k,D_d)\to(k,D_{d-1})$ (see~\cite[Definition~3.6]{GGO}). Thus, we obtain a morphism of graded associative algebras
$$
\Omega_d^{\bullet}\to \Omega_{d-1}^{\bullet},
$$
which commutes with the de Rham differential $\dd$ and whose kernel is the ideal generated by $\Omega$. Thus, the ideal generated by $\Omega$ in $\Omega^{\bullet}_d$ is a $\dd$-ideal.
Further, we have the morphism of graded associative algebras
$$
\Omega_d^{\bullet}\otimes_k A\to \Omega_{d-1}^{\bullet}\otimes_k A,
$$
whose kernel $I^{\bullet}$ is the graded ideal generated by $\Omega$ in $\Omega_d^{\bullet}\otimes_k A$. Since $\dd$ from the CDG-structure on $\Omega_d^{\bullet}\otimes_k A$ satisfies the graded Leibnit rule and the natural homomorphism $\Omega_d^{\bullet}\to \Omega^{\bullet}_d\otimes_k A$ commutes with $\dd$, we deduce that $I^{\bullet}$ is also a $\dd$-ideal. Consequently, $\dd$ induces a map $\dd'$ on the graded associative algebra $\Omega^{\bullet}_{d-1}\otimes_k A$. It follows that this defines a CDG-structure $(\dd',h')$ on $\Omega_{d-1}^{\bullet}\otimes_k A$ with $h'$ being the image of~$h$ under the natural map
$$
\Omega^2_d\otimes_k A\to \Omega^2_{d-1}\otimes_k A.
$$
By the inductive hypothesis, we may assume that $h'=0$, whence $h\in I^2$, where $I^2$ is the second degree part of $I^{\bullet}$.

Put $V:=\Omega\otimes_k A$. Since $\dim_k(\Omega)=1$, we have that
$$
I^i=\Omega^{i-1}_{d-1}\otimes_k V,\,\, i\geqslant 1,\quad\text{and}\quad I^{\bullet}\cdot I^{\bullet}=0.
$$
Since $h\in I^2$, we see that the composition $$\dd\circ \dd=[h,\cdot\:]$$ vanishes on $I^{\bullet}$. We obtain a $(D_{d-1})$-module structure on the finite-dimensional $k$-vector space $V$ with $\nabla_V$ being the map $$\dd:I^1\to I^2.$$ Moreover, the element
$$
h\in \Omega^1_{d-1}\otimes_k V
$$
satisfies $\nabla_V(h)=0$ by the second Bianchi identity (see~Example~\ref{examp-CDGA}\eqref{en:526}).

Since $k$ is linearly $(D_{d-1})$-closed, we see that there is $a\in V$ such that $$\nabla_V(a)=-h,$$ or, equivalently, there is $a\in I^1$ with $\dd(a)=-h$. Since $a\cdot a=0$, the CDG-structure $$(\dd+[a,\cdot],h+\dd(a)+a\cdot a)$$ satisfies the required condition.
\end{proof}

Combining Lemma~\ref{lemma-connCDGA} and Lemma~\ref{lemma-CDGA}, we obtain the following result, which is used for applications to linear $D_k$-groups and isomonodromic parameterized linear differential equations in Sections~\ref{sec:LDAG} and~\ref{sec:IDE}, respectively.

\begin{proposition}\label{prop-main}
Suppose that $(k,D_k)$ is filtered-linearly closed and there is a morphism of differential fields $(k,D_k)\to (K,D_K)$ such that the map $\Omega_k\to \Omega_K$ is injective together with a faithful differential functor $\Cat\to \Vect(K)$. Then there is a $D_k$-connection on an object $X$ in $\Cat$ if and only if there is a $D_k$-structure on $X$.
\end{proposition}

Below in Example~\ref{example-new}, we show that Proposition~\ref{prop-main} is not true over an arbitrary field $(k,D_k)$. The category $\Cat$ in this example is $\Rep(G)$ for a linear $D_k$-group $G$.

Suppose that $\Cat$ is {\it finite}, that is, all Hom-spaces in $\Cat$ are finite-dimensional
over $k$ and all objects have finite length (see \cite{Stalder}).
For example, if $\Cat$ satisfies the condition from Lemma~\ref{lemma-connCDGA}, then it is finite. Let $l$ be a $D_k$-field over $k$. Recall from~\cite{Stalder} that there is an abelian $l$-linear tensor category $l\otimes_k\Cat$, called {\it extension of scalars category}, together with an exact $k$-linear tensor functor
$$
l\otimes_k-:\Cat\to l\otimes_k\Cat.
$$
For short, put $$\Dop:=l\otimes_k\Cat\quad \text{and}\text Y:=l\otimes_k X.$$
By~\cite[Proposition~4.12(i)]{GGO}, there is a canonical $D_l$-structure on $\Dop$ with
$$
\At^1_{\Dop}(Y)\cong l\otimes_k\At_{\Cat}^1(X).
$$
Besides, $\Dop$ is a (not full) subcategory in the category $\Ind(\Cat)$ of ind-objects in $\Cat$ and there is a canonical morphism $X\to Y$ in $\Ind(\Cat)$. For example, if $\Cat=\Vectf(k)$ is the category of finite-dimensional $k$-vector spaces, then $\Dop=\Vectf(l)$ and $l\otimes_k-$ is the usual extension of scalars functor.

\begin{lemma}\label{lemma-main}
In the above notation and assumptions, given a $D_k$-field $l$ over $k$, there is a $D_k$-connection on $X$ in~$\Cat$ if and only if there is a $D_l$-connection on $Y$ in $\Dop$.
\end{lemma}
\begin{proof}
Applying the functor $l\otimes_k-$, we see that a $D_k$-connection on $X$ leads to a $D_l$-connection on $Y$. Conversely, assume that there is a $D_l$-connection $s_Y$ on $Y$.
Choose a $k$-linear map $\lambda:l\to
k$ such that the composition $$k\to
l\stackrel{\lambda}\longrightarrow k$$ is the identity. Then the composition in the category $\Ind(\Cat)$
$$
\begin{CD}
X@>>> Y@>s_Y>> \At^1_{\Dop}(Y)@>\sim>>l\otimes_k \At^1_{\Cat}(X)@>\lambda\otimes\id_{\At^1(X)}>>\At^1_{\Cat}(X)
\end{CD}
$$
defines a $D_k$-connection on $X$ in $\Cat$.
\end{proof}

In general, the $D_k$-connection on $X$ constructed in the proof of Lemma~\ref{lemma-main} can be not a $D_k$-structure. If $\Cat=\Vectf(k)$, then the connection matrices for $X$ are obtained by applying $\lambda$ to the connection matrices for $Y$. Combining Proposition~\ref{prop-main} and Lemma~\ref{lemma-main}, we obtain the following result.

\begin{proposition}\label{prop-descent}
Let $l$ be a $D_k$-field over $k$. Suppose that $(k,D_k)$ is filtered-linearly closed and there is a morphism of differential fields $(k,D_k)\to (K,D_K)$ together with a faithful differential functor $\Cat\to \Vect(K)$. Then there is a $D_k$-structure on $X$ in $\Cat$ if and only if there is a $D_l$-structure on $l\otimes_k X$ in $l\otimes_k\Cat$.
\end{proposition}

\section{Linear differential algebraic groups and conjugation}\label{sec:LDAG}

In this section, we show how  Proposition~\ref{prop-main} can be applied to linear differential algebraic groups. The main results here are in Theorem~\ref{theor-conjug} and Theorem~\ref{thm:redsemisimple}. The behavior of conjugation under extensions of scalars is illustrated in Section~\ref{sec:nonconst}. In particular, Example~\ref{example-PV} shows that the assumption on the ground field made in Theorem~\ref{theor-conjug} cannot be relaxed. Also, Example~\ref{example-new} demonstrates that Proposition~\ref{prop-main} is not true over an arbitrary differential field, and will be further used in Example~\ref{examp-isomon} to justify the need in the filtered-linearly closed assumption in the main result of the paper, Theorem~\ref{thm:54}.

\subsection{Main results}

Let $G$ be a linear $D_k$-group over $k$ and $V$ be a faithful finite-dimensional representation of $G$. Let $A$ be a $D_k$-Hopf algebra over $k$ that corresponds to $G$. A $D_k$-connection on $V$ as an object in $\Cat=\Rep(G)$ (see Definition~\ref{defin-Dkobject}) is a $D_k$-connection on $V$ as a $k$-vector space such that the
coaction map
$$
V\to V\otimes_k A
$$
is a morphism of $k$-vector spaces with $D_k$-connections.
Equivalently, for any $D_k$-algebra $R$, the action of the group $G(R)$ on $R\otimes_k
V$ commutes with the $D_k$-connection.

A $D_k$-connection on $V$ in $\Rep(G)$ is a $D_k$-structure if and only if $V$ is a $D_k$-module. In this case, we also say that $V$ is a {\it $D_k$-representation} of $G$.

\begin{definition}\label{defin-conj}
We say that $G$ {\it is conjugate to a constant subgroup in
$\GL(V)$} if there is a $k_0$-vector space $V_0$ and an isomorphism
$k\otimes_{k_0}V_0\cong V$ of $k$-vector spaces such that there is an embedding in
$\GL(V)$:
$$
G\subset {\GL(V_0)}^c
$$
(see Section~\ref{sec:prelnot} for the definition of ${\GL(V_0)}^c$).
\end{definition}

Note that if $G$ is conjugate to a constant subgroup in $\GL(V)$,
then $G$ is constant: there is an algebraic subgroup $G_0\subset
\GL(V_0)$ such that the isomorphism $$k\otimes_{k_0}V_0\cong V$$
induces the equality $G={(G_0)}^c$ in $\GL(V)$.
We say that $G$ {\it is
conjugate to a reductive constant subgroup in $\GL(V)$} if $G_0$ is reductive.

For an explicit description of Definition~\ref{defin-conj}, choose a
basis in $V$ over $k$. Then $\GL(V)\cong \GL_n(k)$ for some $n$. By
the differential Nullstellensatz (see~\cite[Theorem~IV.2.1]{Kol}), $G$ is conjugate to a constant
subgroup in $\GL(V)$ if and only of there is an element $g\in
\GL_n(k)$ such that, for a $D_k$-closed field $\U$ over $k$
(equivalently, for any $\U$ as above), we have
$$
g^{-1}G(\U)g\subset \GL_n(\U_0),\quad \U_0:=\U^{D_k}.
$$

\begin{example}
Let $k=k_0(t)$, $D_k=k\cdot\partial_t$,
$G\subset \G_a$ be given by the linear equation
$$
\partial_t^2u=0,\quad u\in\G_a,
$$
and let $V$ be a faithful representation of $G$ given by the faithful upper-triangular two-dimensional representation of $\G_a$. Then
$$
G\cong {\left(\G_a^2\right)}^c
$$
is constant, because $\partial_t^2(1)=\partial_t^2(t)=0$.
On the other hand, $G$ is not conjugate to a constant subgroup in $\GL(V)$, because there are no faithful two-dimensional representations of the linear algebraic group $\G^2_a$ over $k_0$. This shows that a constant linear $D_k$-group is not necessarily conjugate to a constant subgroup in $\GL(V)$ for a faithful representation $V$ of $G$.
\end{example}

The following result is also proved
in~\cite[Corollary~1]{OvchRecoverGroup} but only for the case of a
differentially closed field with one derivation.

\begin{proposition}\label{prop-const1}
The $D_k$-group $G$ is conjugate to a constant subgroup in $\GL(V)$ if
and only if there is a $D_k$-structure on $V$ in $\Rep(G)$ such that
$V$ is a trivial $D_k$-module.
\end{proposition}
\begin{proof}
If $G$ is conjugate to a constant subgroup, then the isomorphism
$k\otimes_{k_0}V_0\cong V$ defines a $D_k$-structure on $V$ in
$\Rep(G)$, where $V_0$ is as in Definition~\ref{defin-conj}. Conversely, suppose that we are given a $D_k$-structure on $V$ that satisfies the hypothesis of the proposition. Then put $V_0:=V^{D_k}$.
\end{proof}

Combining Propositions~\ref{prop-main},~\ref{prop-connequiv}, and~\ref{prop-const1}, we obtain the following result.

\begin{theorem}\label{theor-conjug}
Suppose that $(k,D_k)$ is filtered-linearly closed (see Definition~\ref{defin-good}) and linearly $D_k$-closed (see the comment following Definition~\ref{defin-good}). Then $G$ is conjugate to a $D_k$-constant subgroup in
$\GL(V)$ if and only if there is a (possibly, non-commuting) basis
$\partial_1,\ldots,\partial_d$ in $D_k$ over $k$ such that, for all~$i$, $1\Le i\Le d$, $G$
is conjugate to a $\partial_i$-constant subgroup in $\GL(V)$.
\end{theorem}

Let $l$ be a $D_k$-field over $k$, $D_l:=l\otimes_k D_k$. We have $$l\otimes_k\Rep(G)\cong \Rep(G_l).$$ Combining Propositions~\ref{prop-descent} and~\ref{prop-const1}, we obtain the following result.

\begin{proposition}\label{prop-descentgrp}
Suppose that $(k,D_k)$ is filtered-linearly closed and $l$ is linearly $D_l$-closed. Then the linear $D_l$-group $G_l$ is conjugate to a $D_l$-constant subgroup in $\GL(V_l)$ if and only if there is a $D_k$-structure on $V$ in $\Rep(G)$.
\end{proposition}

The following result is also proved
in~\cite[Theorem~3.14]{diffreductive} but just for the case of a
differentially closed field with one derivation.

\begin{theorem}\label{thm:redsemisimple}
Suppose that $(l,D_l)$ is filtered-linearly closed and linearly $D_l$-closed. Then the linear $D_l$-group $G_l$ is conjugate to a reductive constant subgroup in $\GL(V_l)$ if and only if the category $\Rep(G)$ is
semisimple.
\end{theorem}
\begin{proof}
We will use the following fact: given a field extension $E\subset F$, a Hopf algebra $A$ over $E$ corresponds to a reductive linear algebraic group over $E$ if and only if this holds for the extension of scalars $A_F$ over $F$ (see~\cite[Remark~2.1.3(ii)]{Demazure1965}).

Assume that $G_l$ is conjugate to a reductive constant
subgroup in $\GL(V_l)$. By the fact above, this implies that $G$ is a reductive algebraic group over $k$ (with the $D_k$-structure forgotten). Since $\Char k=0$, we obtain that $\Rep(G)$ is semisimple (see~\cite[Chapter~2]{SpringerInv}).

Now assume that $\Rep(G)$ is semisimple. Then there is a $D_k$-connection on $V$ as all exact sequences in $\Rep(G)$ are split. This induces a $D_l$-connection on $V_l$ in $\Rep(G_l)$. By Proposition~\ref{prop-main},
there is a $D_l$-structure on $V_l$ in $\Rep(G_l)$. By
Proposition~\ref{prop-const1}, $G_l$ is conjugate to a constant
subgroup in $\GL(V_l)$. Hence, $k[G]$ is a finitely generated algebra over $k$. Since $\Rep(G)$ is semisimple and $\Char k=0$, we obtain that $G$ is reductive as an algebraic group over $k$. Again, by \cite[Remark~2.1.3(ii)]{Demazure1965}, this implies that $G_l$ is conjugate to a reductive constant subgroup in $\GL(V_l)$.
\end{proof}

\subsection{Examples}\label{sec:nonconst}
First, we provide a non-trivial example to Proposition~\ref{prop-main}.

\begin{example}\label{example-new}
Let $$k:=\Q(t_1,t_2)\quad\text{and}\quad D_k:=k\cdot \partial_{t_1}\oplus k\cdot\partial_{t_2}.$$
Let $V$ be a $3$-dimensional $k$-vector space with a basis $\bar e:=(e_1,e_2,e_3)$. Consider the $D_k$-connection $\nabla_V$ on $V$ given~by
$$
\nabla_V(\bar e):=-\dd t_1\otimes\bar e\cdot B_1-\dd t_2\otimes\bar e\cdot B_2,
$$
where
$$
B_1:=\begin{pmatrix}
0&\frac{1}{t_1}&0\\
0&0&0\\
0&0&0
\end{pmatrix},\quad
B_2:=\begin{pmatrix}
0&0&0\\
0&0&\frac{1}{t_2}\\
0&0&0
\end{pmatrix}.
$$
That is, we have
$$
\partial_{t_i}(\bar e)=-\bar e\cdot B_i.
$$
Note that
\begin{equation}\label{eq:Aicurvature}
\partial_{t_2}B_1-\partial_{t_1}B_2-[B_2,B_1]=\frac{1}{t_1t_2}\cdot \varepsilon,
\quad \text{where}\quad \varepsilon:=
\begin{pmatrix}
0&0&1\\
0&0&0\\
0&0&0
\end{pmatrix}.
\end{equation}
In particular, the $D_k$-connection $\nabla_V$ on $V$ is not a $D_k$-structure on $V$.
Further, consider the unipotent subgroup $\Um$ in $\GL(V)$ that consists of matrices of the following form (with respect to the basis $\bar e$):
$$
g(u_1,u_2,v) := \begin{pmatrix}
1&u_1&v\\
0&1&u_2\\
0&0&1
\end{pmatrix}.
$$
Let $G$ be the linear $D_k$-subgroup in $\Um$ given by the equations
$$
\partial_{t_i} u_j=0,\quad i=1,2,\,j=1,2,
$$
$$
\partial_{t_1} v=\frac{1}{t_1}\cdot u_2,\quad \partial_{t_2} v=-\frac{1}{t_2} u_1.
$$
Note that these equations are equivalent to the equations
$$
\partial_{t_i}g(u_1,u_2,v)+[g(u_1,u_2,v),B_i]=0,\quad i=1,2.
$$
This means that the action of $G$ on $V$ commutes with the action of $D_k$ (see also the discussion following Lemma~\ref{lemma-aut}), that is, $\nabla_V$ is a $D_k$-connection on $V$ as an object in $\Rep(G)$.

Let us show that there is no $D_k$-structure on $V$ in $\Rep(G)$. Assume the converse. By~\eqref{eq:allconnections} (see Section~\ref{sec:categorical}), this means that there exist $C_1,C_2 \in \End_G(V)$ such that
\begin{equation}\label{eq:Aprime}
\partial_{t_2}A_1-\partial_{t_1}A_2-{\left[A_2,A_1\right]}=0,\quad A_i:=B_i + C_i,\ i=1,2.
\end{equation}
A calculation shows that we have an isomorphism (via choosing the basis $\bar e$)
\begin{equation}\label{eq:Endisom}
\End_G(V)\cong k\cdot\Id\oplus\, k\cdot\varepsilon.
\end{equation}
Since
$$
[B_i,\varepsilon]=0,\,i=1,2,
$$
we see that~\eqref{eq:Aicurvature} and~\eqref{eq:Endisom} imply that~\eqref{eq:Aprime} holds if and only if there exist $f_1,f_2\in k$ such that
\begin{equation}\label{eq:f1f2}
\frac{1}{t_1t_2}+\partial_{t_2}f_1-\partial_{t_1}f_2=0.
\end{equation}
This implies that the coefficient of $t_1^{-1}$ with values in $\Q(t_2)$ of the function
$$
\frac{1}{t_1t_2}+\partial_{t_2}f_1
$$
vanishes. Therefore, we have
$$
\frac{1}{t_2}+\partial_{t_2}(a_{-1})=0,\quad\text{where}\quad f_1=\sum_ia_i t_1^i,\ \  a_i\in\Q(t_2).
$$
This gives a contradiction. Thus, we see that Proposition~\ref{prop-main} is not true over an arbitrary field $(k,D_k)$.
\end{example}

Next, we describe two types of $D_k$-subgroups in $\GL_n(k)$ that are not
constant over $k$ but are conjugate to constant subgroups in
$\GL_n(l)$ over $l$, where $l$ is a Picard--Vessiot extension of $k$. Let $M$ be a finite-dimensional $D_k$-module over $k$.


\begin{lemma}\label{lemma-aut}
The group-valued functor
$$
\GL^{D_k}(M):\DAlg(k,D_k)\to \Sets,\quad R\mapsto
\Aut^{D_k}_R(R\otimes_k M)
$$
is represented by a linear $D_k$-group.
\end{lemma}
\begin{proof}
The corresponding finitely generated $D_k$-Hopf algebra is the Hopf
algebra of the algebraic group $\GL(M)$ over $k$ with the
$D_k$-structure obtained by the localization over the determinant of the $D_k$-structure
on the symmetric algebra of the $D_k$-module $\End_k(M)\cong
M^{\vee}\otimes_k M$.
\end{proof}

For an explicit description of $\GL^{D_k}(M)$, choose a basis $\bar
e=(e_1,\ldots,e_n)$ in $M$ over $k$. Then $\GL(M)\cong\GL_n(k)$. For
each $\partial\in D_k$, denote the corresponding connection
$(n\times n)$-matrix by $A_{\partial}$, that is, we have
$$
\partial(\bar e)=-\bar e\cdot A_{\partial}.
$$
By definition, $\GL^{D_k}(M)$ consists of invertible $(n\times
n)$-matrices such that the corresponding gauge transformation
preserves the connection matrices $A_{\partial}$ for all
$\partial\in D_k$. Thus, $\GL^{D_k}(M)$ is given by the differential
equations
$$
\partial g+[g,A_{\partial}]=0,\quad g\in \GL_n(k),\,\partial\in D_k.
$$

Note that $\GL^{D_k}(M)$ is a closed linear $D_k$-subgroup in the
linear $D_k$-group $\GL(M)$ and $M$ is faithful a $D_k$-representation of $\GL^{D_k}(M)$. A morphism of linear $D_k$-groups $$G\to \GL^{D_k}(M)$$ corresponds to a $D_k$-representation $V$ of $G$ such that $V\cong M$ as $D_k$-modules.

It follows directly from the proof of Lemma~\ref{lemma-aut} that the
linear $D_k$-group $\GL^{D_k}(M)$ is constant if and only if the
$D_k$-module $$M^{\vee}\otimes_k M$$ is trivial, because a submodule
of a trivial $D_k$-module is trivial and the determinant is a $D_k$-constant in $\Sym^n_k(M^\vee\otimes_k M)$. Besides, if $M$ is trivial, then $\GL^{D_k}(M)$ is conjugate to a constant subgroup in $\GL(M)$ (the converse is not true already for $\dim_k(M)=1$).

\begin{example}\label{example-PV}
Consider the differential field
$
k=\Q(t_1,t_2)$, $D_k:=k\cdot\partial_{t_1}\oplus
k\cdot\partial_{t_2}$
and the $D_k$-module $M:=\uno\oplus L$, where
$
L=k\cdot e$, $\partial_{t_1}(e)=0$, $\partial_{t_2}(e)=e$.
Since the $D_k$-module
$$
M^{\vee}\otimes_k M\cong\uno\oplus L\oplus L^{\vee}\oplus\uno
$$
is not trivial, the linear $D_k$-group
$$
G:=\GL^{D_k}(M)\subset \GL_2(k)
$$
is not constant and henceforth $G$ is
not conjugate to a constant subgroup in $\GL_2(k)$. Put
$$
\partial_1:=t_1\partial_{t_1},\quad
\partial_2:=t_1\partial_{t_1}+\partial_{t_2}.
$$
Then $[\partial_1,\partial_2]=0$ and $L$ is a trivial
$\partial_2$-module as $$\partial_2{\left(t_1^{-1}\cdot e\right)}=0.$$
Therefore, $M$ is a trivial $\partial_i$-module for $i=1,2$. By Proposition~\ref{prop-const1}, $G$ is
conjugate to a $\partial_i$-constant subgroup in $\GL_2(k)$ separately with respect to each $i$. This shows that Theorem~\ref{theor-conjug} is not
true for an arbitrary $(k,D_k)$.
\end{example}

The following type of non-constant groups will be used in Section~\ref{sec:IDE} in order to construct non-trivial examples to Theorem~\ref{theor-main}.

\begin{lemma}\label{lemma-invmod}
The group valued functor
$$
M^{D_k}:\DAlg(k,D_k)\to\Sets,\quad R\mapsto (R\otimes_k M)^{D_k}
$$
is represented by a linear $D_k$-group.
\end{lemma}
\begin{proof}
The corresponding finitely generated $D_k$-Hopf algebra is the Hopf
algebra of $(M,+)$, that is, the symmetric algebra of $M^\vee$, with
the induced $D_k$-structure (see also \cite[Lemma~2.16]{Michael}).
\end{proof}

For an explicit description of $M^{D_k}$, choose a basis in $M$ over
$k$. Then $(M,+)\cong\G_a^n$ for some $n$. For each $\partial\in
D_k$, denote the corresponding connection $(n\times n)$-matrix by
$A_{\partial}$. Then $M^{D_k}$ is given by the differential
equations
$$
\partial y=A_{\partial}\cdot y,\quad y\in \G_a^n,\,\partial\in D_k.
$$

Note that $M^{D_k}$ is a closed linear $D_k$-subgroup in the linear
$D_k$-group $M$. One can show that a linear $D_k$-group $G$ is isomorphic to ${(\G_a^n)}^c$
over some $D_k$-field $l$ over $k$ if and only if $G\cong M^{D_k}$ for some $n$-dimensional $D_k$-module $M$ over $k$ (see also~\cite[Proposition~11]{Cassidy}).

Assume that $D_k=k\cdot\partial$ for a derivation $\partial:k\to k$
and $k_0\ne k$. Let $m\in M$ be a cyclic vector (see \cite[Definition~2.8]{Michael}). Let $0\ne D$ be a
linear $\partial$-operator with coefficients in $k$ of the smallest order such that $Dm=0$, which exists because $M$ is finite-dimensional over $k$. Then
$$
{\left(M^\vee\right)}^{D_k}
$$
is isomorphic to the $D_k$-subgroup in $\G_a$ given by the equation
$
Du=0$, $u\in \G_a$,
because
they represent the same functor (see~\cite[Lemma~2.16]{Michael}).

\begin{remark}\label{remark-nonconstGa}
\hspace{0cm}
\begin{enumerate}
\item\label{en:719}
There is a faithful $D_k$-representation $V_M$ of $M^{D_k}$ defined
as follows: as a $D_k$-module, $V_M$ is $M\oplus\uno$, and the action
of $M^{D_k}$ is given by
$$
m:(n,c)\mapsto(n+c\cdot m,c),
$$
where $m\in M^{D_k}$, $n\in M$, and $c\in k$. We have an exact sequence of
$D_k$-representations of $M^{D_k}$
$$
\begin{CD}
0@>>> M@>>> V_M@>>> \uno@>>> 0,
\end{CD}
$$
where $M^{D_k}$ acts trivially on the $D_k$-modules $M$ and $\uno$.
\item\label{it:740}
One can show that there is a bijection between morphisms of linear $D_k$-groups $G\to M^{D_k}$ and isomorphism classes of exact sequences of $D_k$-representations of $G$
$$
\begin{CD}
0@>>> M@>>> V@>>> \uno@>>> 0,
\end{CD}
$$
where $G$ acts trivially on the $D_k$-modules $M$ and $\uno$. An argument shows that this
implies that, for a linear $D_k$-group $G$, the category $\Rep(G)$ is
$D_k$-equivalent to $\Rep{\left(M^{D_k}\right)}$ if and only if
$G\cong M^{D_k}$.
\end{enumerate}
\end{remark}

\section{Gauss--Manin connection and parameterized differential Galois groups}\label{sec:GM}

The main results of this section, Propositions~\ref{prop-intGalois} and~\ref{prop-PPV}, are used in Section~\ref{subsection-exampisom} in order to construct non-trivial examples to Theorem~\ref{theor-main}. The constructions and results of this section seem to have also their own interest in the parameterized differential Galois theory.

\subsection{Gauss--Manin connection}\label{subsection-GM}

We define algebraically a Gauss--Manin connection, which is used to describe a parameterized differential Galois group of integrals in Section~\ref{subsection-integrals}. For this, we use the Gauss--Manin connection on $H^1$ only, so that the reader may put $i=1$ in what follows if desired.

For any differential field $(K,D_K)$, let $H^i(K,D_{K})$ denote the cohomology groups of the de Rham complex $\Omega^{\bullet}_K$ (see Section~\ref{sec:prelnot}). That is, we have
$$
H^i(K,D_K):=\Ker{\big(\Omega^i_K\stackrel{\dd}\longrightarrow \Omega^{i+1}_{K}\big)}\big/\im{\big(\Omega^{i-1}_K\stackrel{\dd}\longrightarrow
\Omega^i_{K}\big)},\ i \Ge 1,
$$
and $H^0(K,D_K)=K^{D_K}$. Recall that, for $\partial\in D_K$, the Lie derivative is defined as follows (see~\cite[\S3.10]{GGO}):
$$
L_{\partial}=\dd\circ i_{\partial}+i_{\partial}\circ\dd:\Omega^i_K\to\Omega^i_K,
$$
where
$$
i_{\partial}:\Omega^i_K\to\Omega^{i-1}_K,\quad \omega\mapsto\left\{a\mapsto\omega(\partial\wedge a),\quad a\in\wedge^{i-1}_K D_K\right\},\ i \Ge 1
$$
and $i_{\partial}=0$ for $i=0$. In particular,
$$
L_{\partial}(a)=\partial(a)\quad \text{for any}\ \ a\in K.
$$
It follows from the definition that the Lie derivative commutes with $\dd$, acts as zero on $H^i(K,D_K)$, satisfies the Leibniz rule
$$
L_{\partial}(\omega\wedge\eta)=L_{\partial}(\omega)\wedge\eta+\omega\wedge L_{\partial}(\eta)
$$
for all $\omega\in\Omega_K^i$, $\eta\in \Omega^j_K$, $i,j\Ge0$, and we have $$L_{[\partial,\delta]}=[L_{\partial},L_{\delta}]\quad \text{for all}\ \ \partial,\delta\in D_K.$$
Let now $(K,D_K)$ be a parameterized differential field over $(k,D_k)$. We have the relative de Rham complex $\Omega^{\bullet}_{K/k}$, where $$\Omega_{K/k}:=D_{K/k}^{\vee}.$$ For short, put
$$
H^i(K/k):=H^i(K,D_{K/k}),\ \ i\geqslant 0.
$$
Then $H^i(K/k)$ are $k$-vector spaces, because the differential on $\Omega^{\bullet}_{K/k}$ is $k$-linear. Moreover, there is a canonical $D_k$-structure on $H^i(K/k)$, called a {\it Gauss--Manin connection} and constructed as follows.

For $\partial\in D_k$, let $\tilde\partial\in D_K$ be any lift of $1\otimes\partial$ with respect to the structure map $D_K\to K\otimes_k D_k$. One checks that the action of $L_{\tilde\partial}$ on $\Omega_K$ preserves $\Omega_k$. Since the kernel
$$
C^{\bullet}:=\Ker\big(\Omega_K^{\bullet}\to\Omega_{K/k}^{\bullet}\big)
$$
is generated by $\Omega_k$ as an ideal in $\Omega^{\bullet}_K$ with respect to the wedge product and $L_{\tilde\partial}$ satisfies the Leibniz rule as mentioned above, we see that the action of $L_{\tilde\partial}$ on $\Omega^{\bullet}_K$ preserves the subcomplex $C^{\bullet}$. Therefore, $L_{\tilde\partial}$ is well-defined on the quotient $\Omega^{\bullet}_{K/k}$. Since $D_{K/k}$ acts as zero on $H^i(K/k)$, we obtain a well-defined action of $D_k$ on $H^i(K/k)$. Finally, one checks that the corresponding map $$D_k\to \End_{\Z}{\left(H^i(K/k)\right)}$$ is $k$-linear, whence $H^i(K/k)$ is a $D_k$-module.

Explicitly, for any $\omega\in\Omega^i_{K/k}$ with $\dd\omega=0$, we have
\begin{equation}\label{eq:GM}
\partial[\omega]={\left[L_{\tilde\partial}(\tilde\omega)\right]},
\end{equation}
where $\tilde\omega\in\Omega^i_K$ is any lift of $\omega$ with respect to the map $\Omega^i_K\to\Omega^i_{K/k}$, and the brackets denote taking the class in $H^i(K/k)$. The preceding discussion shows that $\partial[\omega]$ is well-defined.


\begin{example}\label{examp-GM}
\hspace{0cm}
\begin{enumerate}
\item
We have $H^0(K/k)=K^{D_{K/k}}=k$ with the usual $D_k$-structure.
\item\label{en:770}
Suppose that $D_{K/k}=K\cdot\partial_x$. Then $$\Omega_{K/k}=\Omega_{K/k}^1=K\cdot\omega_x\quad \text{with}\quad \omega_x(\partial_x)=1,\ \ \dd a=\partial_x(a)\cdot\omega_x$$ for any $a\in K$, and $\Omega^i_{K/k}=0$ for $i\Ge 2$. Hence, there is an isomorphism
$$
K/(\partial_x K)\stackrel{\sim}\longrightarrow H^1(K/k),\quad [a]\mapsto [a\cdot\omega_x],
$$
where $a\in K$. Under the above isomorphism, the Gauss--Manin connection on $H^1(K/k)$ corresponds to the $D_k$-structure on $K/(\partial_x K)$ given by $$\partial[a]=\big[\tilde\partial(a)\big],\quad \partial\in D_k,$$
where, as above, $\tilde\partial\in D_K$ is any lift of $1\otimes\partial$ with respect to the structure map $D_K\to K\otimes_k D_k$.
\item\label{en:781}
Suppose, in addition to~\eqref{en:770}, that $K=k(x)$ and $\partial_x(x)=1$. Then, for any class in $K/(\partial_x K)$, there is a unique representative of the form
$$
\sum_{i=1}^n \frac{b_i}{x-c_i}, \quad b_i,c_i\in k,\,b_i\ne 0.
$$
Since
$$
\partial\left[\sum_{i=1}^n \frac{b_i}{x-c_i}\right]=\left[\sum_{i=1}^n \frac{\partial b_i}{x-c_i}\right],
$$
we obtain isomorphisms of $D_k$-modules
$$
\mbox{$\bigoplus\limits_{c\in k}k$}\cong K/(\partial_x K)\cong H^1(K/k).
$$
\end{enumerate}
\end{example}

\subsection{PPV extensions defined by integrals}\label{subsection-integrals}

As above, let $(K,D_K)$ be a parameterized differential field over $(k,D_k)$.
Given  $\omega\in\Omega_{K/k}$ with $\dd \omega=0$, the
equation $\dd y=\omega$ corresponds to a consistent system of
(non-homogenous) linear differential equations in the unknown $y$
$$
\delta(y)=\omega(\delta),\quad\delta\in D_{K/k}.
$$
Note that Lemma~\ref{lemma-invmod} remains valid if one assumes that $M$ is a $D_k$-finitely generated module over $k$ instead of being finite-dimensional over $k$. We use this generality in the following statement. Its special case appears in \cite[Lemma~2.3]{MichaelLAG}.

\begin{proposition}\label{prop-intGalois}
Let $L$ be a PPV extension of $K$ for the system of linear
differential equations that corresponds to the equation $\dd
y=\omega$, where $\omega\in \Omega_K$ with $\dd\omega=0$ (see above). Let $M$ be
the $D_k$-submodule in $H^1(K/k)$ generated by $[\omega]$ (see Section~\ref{subsection-GM}). Then
there is an isomorphism of linear $D_k$-groups (see Lemma~\ref{lemma-invmod} and the remark preceding the proposition)
$$
\Gal^{D_K}(L/K)\cong {\left(M^\vee\right)}^{D_k}.
$$
\end{proposition}
\begin{proof}
The proof is in the spirit of the Kummer and Artin--Schreier
theories, for example, see~\cite[\S VI.8]{LangAlgebra}. Let $R$ be a $D_k$-algebra. The natural map
$$
\alpha:R\otimes_k H^1(K/k)\to R\otimes_k H^1(L/k)
$$
is a morphism of $D_k$-modules. Since $L$ contains a solution of the equation $\dd y=\omega$, we have $\alpha([\omega])=0$. Therefore, for any $\eta\in R\otimes_k\Omega_{K/k}$ with
$\dd\eta=0$ and $[\eta]\in R\otimes_kM$, we have
$$
\alpha([\eta])=0.
$$
Thus, the equation $\dd y=\eta$ has a
solution in $R\otimes_k L$. Let $\int \eta\in R\otimes_k L$ denote any of these solutions. For each $g\in \Gal^{D_K}(L/K)(R)$, consider the map
$$
\mbox{$\phi_g:R\otimes_k M\to R,\quad [\eta]\mapsto g{\left(\int \eta\right)}-\int \eta.$}
$$
One checks that $\phi_g([\eta])$ is well-defined, that is, does not depend on the choices of $\eta$ and $\int\eta$ for a given $[\eta]$, and belongs~to
$$
R=(R\otimes_k L)^{D_{K/k}}.
$$
Further, $$\phi_g\in {\left(M^{\vee}\right)}^{D_k},$$ that is, $\phi_g$ is a $D_k$-map: for any $\partial\in D_k$ and its lift $\tilde\partial\in D_K$, we have
$$
\mbox{$\partial{\left(\phi_g([\eta])\right)}=\tilde\partial{\left(g(\int\eta)\right)}- \tilde\partial{\left(\int\eta\right)}=g\big(\tilde\partial(\int\eta)\big)-
\tilde\partial{\left(\int\eta\right)}=
g{\left(\int L_{\tilde\partial}\eta\right)}-\int L_{\tilde\partial}\eta=\phi_g(\partial[\eta])$},
$$
because the restriction of $\tilde\partial$ from $R\otimes_k L$ to $R$ is $\partial$, $g$ commutes with $\tilde\partial$, $\dd$ commutes with $L_{\tilde\partial}$, and by~\eqref{eq:GM} (see Section~\ref{subsection-GM}).
Also,  for all $g,h\in \Gal^{D_K}(L/K)(R)$, we have
$$
\mbox{$hg(\int\eta)-h(\int\eta)=g(\int\eta)-\int\eta$},
$$
as the right-hand side belongs to $R$ and is Galois invariant. Therefore, $$\phi_{hg}=\phi_h+\phi_g.$$
Summing up, we obtain a morphism of linear $D_k$-groups
$$
\phi:\Gal^{D_K}(L/K)\to {\left(M^\vee\right)}^{D_k}.
$$
Since $L$ is $D_K$-generated over $K$ by $\int\omega$, we see that $\phi$ is injective. Suppose that $\phi$ is not surjective. Then there is a non-zero element $[\eta]\in M$ such that, for any $D_k$-algebra $R$ over $k$ and any $g\in \Gal^{D_K}(L/K)(R)$, we have
$$
\phi_g([\eta])=0.
$$
Equivalently, for any $g\in \Gal^{D_K}(L/K)(R)$, we have
$$
\mbox{$g{\left(\int\eta\right)}=\int\eta$},
$$
whence $\int\eta\in K$ and $[\eta]=0$ in $H^1(K/k)$, which is a contradiction. Thus, $\phi$ is an isomorphism.
\end{proof}

The fact that the parameterized differential Galois group in Proposition~\ref{prop-intGalois} does not depend of the PPV extension corresponds directly to Remark~\ref{remark-nonconstGa}\eqref{it:740}.

\begin{example}~\label{examp-paramGG}
\hspace{0cm}
\begin{enumerate}
\item\label{remark-expliintGalois}
Assume that $D_k=k\cdot\partial_t$.
Let $0\ne D$ be a linear $\partial_t$-operator with coefficients in $k$ of the smallest order such that $$D[\omega]=0\ \ \text{in}\ \ H^1(K/k).$$ If there is no non-zero $D$ with $D[\omega]=0$, then we put $D:=0$. Proposition~\ref{prop-intGalois} and the discussion following Lemma~\ref{lemma-invmod} imply that $\Gal^{D_K}(L/K)$ is isomorphic to
the $D_k$-subgroup in $\G_a$ given by the equation
$$
Du=0,\ \ u\in\G_a.
$$
\item
We use the notation of Example~\ref{examp-GM}~\eqref{en:781}. By Proposition~\ref{prop-intGalois}, the parameterized differential Galois group of the equation $\dd y=\omega$ with
$$
\omega=\sum_{i=1}^n \frac{b_i}{x-c_i}\cdot\omega_x, \quad b_i,c_i\in k,\,b_i\ne 0,
$$
is isomorphic to ${(\G^n_a)}^c$. This is also explained in~\cite[Example~7.1]{PhyllisMichael}.
\end{enumerate}
\end{example}

Surprisingly, the description of the parameterized differential Galois group given in Proposition~\ref{prop-intGalois} allows to prove the existence of a PPV extension. For simplicity, suppose that $D_k=k\cdot\partial_t$ and let $D$ be as in Example~\ref{examp-paramGG}\eqref{remark-expliintGalois}. Let $\tilde\partial_t\in D_K$ be a lift of $1\otimes \partial_t$ with respect to the structure map $D_K\to K\otimes_k D_k$ and let a linear $D_K$-operator $\widetilde{D}$ be the corresponding lift of $D$. Since $D[\omega]=0$, there is $a\in K$ such that
\begin{equation}\label{eq:GM2}
L_{\widetilde{D}}(\omega)=\dd a
\end{equation}
by~\eqref{eq:GM} and the preceding discussion (see Section~\ref{subsection-GM}). An equation similar to~\eqref{eq:GM2} was considered in~\cite{MichaelLAG} and~\cite{CKS2012}. Consider the $D_K$-algebra
\begin{equation}
R:=K\{y\}\big/{\big(\dd y-\omega,\widetilde{D}y-a\big)_{D_K}},
\end{equation}
where $(\Sigma)_{D_K}$ denotes the $D_K$-differential ideal generated by $\Sigma$, and $\dd y-\omega$ means the collection
$$
\delta(y)-\omega(\delta),\quad\delta\in D_{K/k}.
$$
Note that $R$ is isomorphic as a $K$-algebra to the ring of polynomials over $K$ (possibly, of countably many variables). In particular, $R$ is a domain.

\begin{proposition}\label{prop-PPV}
In the above notation, the field $L:=\Frac(R)$ is a PPV extension of $K$ for the equation $\dd y=\omega$.
\end{proposition}
\begin{proof}
Let $l$ be a $D_k$-field over $k$ and suppose that the proposition is true for the parameterized field $$K_l:=\Frac(l\otimes_k K)$$ over $l$ (see~\cite[\S8.2]{GGO} for the extension of $D_{K/k}$-constants in parameterized differential fields). That is, suppose that
$$
L_l:=\Frac(l\otimes_k R)
$$
is a PPV extension of $K_l$ for the equation $\dd y =\omega$. Therefore, $L_l^{D_{K/k}}=l$. On the other hand, by~\cite[Corollary 8.9]{GGO}, we have that $$L_l^{D_{K/k}}=l\otimes_k L^{D_{K/k}},$$ whence $L^{D_{K/k}}=k$ and we obtain the needed result for $L$. Thus, we may assume that $(k,D_k)$ is differentially closed.

Now suppose that the proposition is true for $a$ and let $a'\in K$ be another element such that
$$
L_{\widetilde{D}}(\omega)=\dd a'.
$$
Then $a'=a+b$ with $b\in k$. Since $(k,D_k)$ is differentially closed, there is $c\in k$ such that $Dc=b$. This defines an isomorphism
$$
R\to K\{y\}\big/{\big(\dd y-\omega,\widetilde{D}y-a'\big)}_{D_K},\quad y\mapsto y-c.
$$
Thus, it is enough to show that there is at least one $a\in K$ with $L_{\widetilde D}(\omega)=\dd a$ such that the proposition is true for $a$.

Again, since $(k,D_k)$ is differentially closed, there is a PPV extension $E$ of $K$ for the equation $\dd y=\omega$ by~\cite[Theorem~3.5(1)]{PhyllisMichael}. Let $z\in E$ be a solution of the latter equation. Consider the subring $S$ in $E$ that is $D_K$-generated by $z$. We have that
$$
L_{\widetilde{D}}(\omega)=\dd{\big(\widetilde{D}z\big)}.
$$
Since $E^{D_{K/k}}=k$, we see that $\widetilde{D}z\in K$. Put $$a:=\widetilde{D}z.$$ Then we obtain a surjective $D_K$-morphism $f:R\to S$ sending $y$ to $z$.

By Proposition~\ref{prop-intGalois} and Example~\ref{examp-paramGG}\eqref{remark-expliintGalois}, $\Gal^{D_K}(E/K)$ is isomorphic to the $D_k$-subgroup in $\G_a$ given by 
$$
Du=0,\ \ u\in\G_a.
$$
It follows from the proof of Proposition~\ref{prop-intGalois} that the action of $\Gal^{D_K}(E/K)$ on $E$ is given by the formula
$$
u:z\mapsto z+u.
$$
Let $G$ be the extension of scalars from $k$ to $K$ of $\Gal^{D_K}(E/K)$ as a (pro-)algebraic group over $k$. It follows from the PPV theory that $\Spec(S)$ is a torsor under $G$ over $K$ (see~\cite[\S9.4]{PhyllisMichael}). By the explicit description of $R$, $\Spec(R)$ is also a torsor under $G$ and $f$ corresponds to a closed embedding $\Spec(S)\to\Spec(R)$ of $G$-torsors. We conclude that $f$ is an isomorphism, which proves the proposition for the above choice of~$a$.
\end{proof}

\section{Isomonodromic differential equations}\label{sec:IDE}

In this section, we show how Proposition~\ref{prop-main} can be applied to isomonodromic parameterized linear differential equations. The main results here are in Theorem~\ref{thm:54} and Theorem~\ref{theor-main}. Section~\ref{sec:anal} provides an analytic interpretation of our results. The main illustrating examples are in Section~\ref{subsection-exampisom} (see also Example~\ref{ex:replaceBis}).

\subsection{Main results}\label{sec:61}

Let $(K,D_K)$ be a parameterized differential field over $(k,D_k)$ and $N$ be a finite-dimensional $D_{K/k}$-module over~$K$.

\begin{definition}\label{defin-isomon}
We say that $N$ is {\it isomonodromic} if there is a $D_K$-structure on $N$ such that its restriction from $D_K$ to $D_{K/k}$ is equal to the initial $D_{K/k}$-structure on $N$.
\end{definition}

This is called complete integrability in~\cite[Definition~3.8]{PhyllisMichael}, but we preferred to use the terminology slightly more common in differential equations for this notion (see also Section~\ref{sec:anal}).

\begin{proposition}\label{prop-2}
A finite-dimensional $D_{K/k}$-module $N$ is isomonodromic if and
only if there is a $D_k$-structure on $N$ in $\DMod{\left(K,D_{K/k}\right)}$.
\end{proposition}
\begin{proof}
We use facts about the Atiyah functor $\At^1$ in $\DMod\left(K,D_{K/k}\right)$ that can be found in~\cite[\S5.1]{GGO}. We have the equality of sets (see \cite[eq.~(17)]{GGO})
$$
\At^1(N)={\left\{n\otimes
1+\sum\nolimits_i n_i\otimes\omega_i\in N\oplus\left(N\otimes_K\Omega_K\right)\:\big|\:\forall \delta\in D_{K/k},\:
\delta(n)=\sum\nolimits_i\omega_i(\delta)n_i\right\}}
$$
(we are not specifying  $K$-linear and $D_K$ structures on $\At^1(N)$ here).
Further,
$$
\At^1(N)\subset \At^1_K(N),$$
where $\At^1_K$ denotes the Atiyah functor in $\Vect(K)$. Assume that there is a $D_k$-structure $s_N:N\to \At^1(N)$ on $N$ in $\DMod{\left(K,D_{K/k}\right)}$. Since the forgetful functor $$\DMod{\left(K,D_{K/k}\right)}\to \Vect(K)$$ is differential (see~\cite[Theorem~5.1]{GGO}), the composition
$$
\begin{CD}
N@>s_N>>\At^1(N)@>>>\At^1_K(N)
\end{CD}
$$
defines a $D_K$-structure on $N$ that extends  the given $D_{K/k}$-structure. Conversely, assume that $N$ is isomonodromic. Since the $D_K$-structure extends the given $D_{K/k}$-structure, we see that the map $N\to \At^1_K(N)$ factors through $\At^1(N)$ by the construction of $\At^1$, which gives the needed splitting $s_N$.
\end{proof}

For each $\partial\in D_k$, we have a parameterized differential field $(K,D_{K,\partial})$ over $(k,D_k)$ with $D_{K,\partial}$ being the preimage of $K\otimes\partial$ with respect to the structure map $D_K\to K\otimes_k D_k$. By definition, a finite-dimensional $D_{K/k}$-module $N$ is $\partial$-isomonodromic if and only if it is isomonodromic over $(K,D_{K,\partial})$. By Proposition~\ref{prop-2}, this is equivalent to the existence of a $\partial$-structure on $N$ in $\DMod(K,D_{K/k})$. Note that we have a morphism of differential fields $(k,D_k)\to (K,D_K)$ (while there is no fixed $D_k$-field structure on $K$) and the forgetful functor $\DMod(K,D_{K/k})\to\Vect(K)$ is a faithful differential functor (see~\cite[Theorem~5.1]{GGO}). Thus, combining Proposition~\ref{prop-main},~\ref{prop-connequiv}, and~\ref{prop-2}, we obtain:

\begin{theorem}\label{thm:54}
Suppose that $(k,D_k)$ is filtered-linearly closed. Then $N$ is isomonodromic if
and only if there is a (possibly, non-commuting) basis $\partial_1,\ldots,\partial_d$ in $D_k$ over $k$ such that $N$ is $\partial_i$-isomonodromic for all $i$.
\end{theorem}


\subsection{Explicit approach}\label{sec:explicit}

Let us explain Theorem~\ref{thm:54} more explicitly in the case mentioned in the introduction: $\dim_K(D_{K/k})=1$ and $\dim_k(D_k)=d$. More precisely, let $\partial,\partial_1,\ldots\partial_d$ denote commuting derivations from $K$ to itself, $k=K^{\partial}$, and put
$$
D_K:=K\cdot\partial\oplus K\cdot\partial_1\oplus\ldots\oplus K\cdot\partial_d\quad\text{and}\quad
D_k:=k\cdot\partial_1\oplus\ldots\oplus k\cdot\partial_d.
$$
Choosing a basis in a finite-dimensional $D_{K/k}$-module $N$ over $K$, we obtain a correspondence between differential structures on $N$ and matrices with entries from $K$. Let $\Mn_n(K)$ denote the space of $(n\times n)$-matrices with entries in $K$. A particular case of Theorem~\ref{thm:54} reads as follows.

\begin{theorem}\label{theor-expl}
Suppose that for all $i$, $1\Le i\Le d-1$, all consistent systems of linear differential equations over $k$ that involve only the derivation $\partial_{1},\ldots,\partial_{i}$ have a fundamental solution matrix with entries in $k$. Let $A\in \Mn_n(K)$ and suppose that there exist matrices $B_1,\ldots,B_d\in \Mn_n(K)$ that satisfy
\begin{equation}\label{eq:Bi}
\partial_{i}A-\partial B_i=[B_i,A]
\end{equation}
for all $i$, $1\Le i\Le d$. Then there exist matrices $A_1,\ldots,A_d\in
\Mn_n(K)$ such that
\begin{equation}\label{eq:Ai1}
\partial_{i}A-\partial A_i=[A_i,A]
\end{equation}
for all $i$, $1\Le i\Le d$, and, for all $i,j$, $1\Le i,j\Le d$, we have
\begin{equation}\label{eq:Ai2}
\partial_{i}A_j-\partial_{j} A_i=[A_i,A_j].
\end{equation}
\end{theorem}

The condition on the field $k$ and derivations $\partial_1,\ldots\partial_d$ from Theorem~\ref{theor-expl} corresponds to the condition that $(k,D_k)$ is filtered-linearly closed with respect to the basis $\partial_1,\ldots,\partial_d$. The explicit proof below can be used in designing algorithms.

\begin{proof}
What follows is an explicit version of the proof of Lemma~\ref{lemma-CDGA}. Similarly, we use induction on $d$, with the case $d=1$ being trivial. Let us make the inductive step from $d-1$ to $d$, $d\Ge 2$. By the inductive hypothesis, there are matrices $A_1,\ldots, A_{d-1}\in\Mn_n(K)$ that satisfy both~\eqref{eq:Ai1} and~\eqref{eq:Ai2}.
We claim that there exists $C\in \Mn_n(K)$ such that the matrices
\begin{equation*}\label{eq:Di}
(A_1,\ldots,A_{d-1},B_d+C)
\end{equation*}
satisfy both~\eqref{eq:Ai1} and~\eqref{eq:Ai2}. In order to show that $B_d+C$ satisfies~\eqref{eq:Ai1} and~\eqref{eq:Ai2}, we need to show the equalities
\begin{equation}\label{eq:55}
\partial(B_d+C)-\partial_{d}A=[A,B_d+C]
\end{equation}
and
\begin{equation}\label{eq:6}
\partial_{i}(B_d+C)-\partial_{d}A_i=[A_i,B_d+C]
\end{equation}
for all $i$, $1\Le i \Le d-1$.
Expanding the left-hand side of~\eqref{eq:55} using~\eqref{eq:Bi}, we see that
\begin{align*}
\partial(B_d+C)&=\partial B_d+\partial C=\partial_{d} A+[A,B_d]+\partial C.
\end{align*}
Consider
\begin{equation}\label{eq:main}
\partial Z=[A,Z]
\end{equation}
as a matrix linear differential equation in $(n\times n)$-matrix $Z$.
Rearranging the terms in~\eqref{eq:6}, we see that we need to find $C\in
\Mn_n(K)$ such that $C$ satisfies~\eqref{eq:main}
and the following condition is satisfied:
\begin{equation}\label{eq:4}
\partial_{i} C+[C,A_i]=\partial_{d} A_i-\partial_{i} B_d+[A_i,B_d]
\end{equation}
for all $i$, $1\Le i\Le d-1$.

We now show that the right-hand side of~\eqref{eq:4} satisfies~\eqref{eq:main}. Indeed, we have:
\begin{align*}
&\partial\big(\partial_{d} A_i-\partial_{i}B_d+[A_i,B_d]\big)= \partial_{d}(\partial A_i)-\partial_{i}(\partial B_d)+[\partial A_i,B_d]+[A_i,\partial B_d]=\\
&=\partial_{d}(\partial_{i} A+[A,A_i])-\partial_{i}(\partial_{d}A+[A,B_d])+[\partial_{i}(A)+[A,A_i],B_d]+[A_i,\partial_{d}A+[A,B_d]]=\\
&=\partial_{d}([A,A_i])-\partial_{i}([A,B_d])+[\partial_{i}(A),B_d]+[[A,A_i],B_d]+[A_i,\partial_{d}(A)]+[A_i,[A,B_d]]=\\
&=[A,\partial_{d}(A_i)]-[A,\partial_{i}(B_d)]+[A,[A_i,B_d]],
\end{align*}
as desired. Here, we have used~\eqref{eq:Bi} and~\eqref{eq:Ai1} for the second equality.

For any matrix $Z\in \Mn_n(K)$ satisfying~\eqref{eq:main}, we now show that for all $i$, $1\Le i\Le d-1$, the matrix
$$
\partial_{i}Z+[Z,A_i]
$$
also satisfies~\eqref{eq:main}. Indeed, we have:
\begin{align*}
\partial\left(\partial_{i}Z+[Z,A_i]\right)&=\partial_{i}\left(\partial Z \right)+\left[\partial Z,A_i\right]+\left[Z,\partial A_i\right]=\\
&=\partial_{i}([A,Z])+[[A,Z],A_i]+\left[Z,\partial_{i}A+[A,A_i]\right]=\\
&=[\partial_{i} A,Z]+[A,\partial_i Z]+[[A,Z],A_i]+\left[Z,\partial_{i}A\right]+[[Z,A],A_i]+[A,[Z,A_i]]=\\
&=\left[A,\partial_{i}Z+[Z,A_i]\right],
\end{align*}
as desired. Here, we have used~\eqref{eq:main} and~\eqref{eq:Ai1} for the second equality as well.

Let $V$ denote the set of all matrices in $\Mn_n(K)$ satisfying~\eqref{eq:main}. Then $V$ is a finite-dimensional $k$-vector space. Also, consider the maps
$$
\Phi_i:\Mn_n(K)\to \Mn_n(K),\quad Z\mapsto \partial_iZ+[Z,A_i],
$$
which are given by the left-hand side of~\eqref{eq:4}. What we have shown so far is that for all $i$, $1\Le i\Le d-1$, the right-hand side of~\eqref{eq:4} belongs to $V$ and that the maps $\Phi_i$ preserve the subspace $V\subset\Mn_n(K)$. Moreover, the maps $\Phi_i$ define a $(\partial_1,\ldots,\partial_{d-1})$-module structure on $V$, that is, these maps satisfy the Leibniz rule and the integrability conditions $[\Phi_i,\Phi_j]=0$ for all $i,j$, $1\Le i,j\Le d-1$. Indeed, we have:
\begin{align*}
(\Phi_i\Phi_j)(Z)&=\partial_i(\partial_jZ+[Z,A_j])+[\partial_jZ+[Z,A_j],A_i]=\\
&
=\partial_i\partial_jZ+[\partial_iZ,A_j]+[Z,\partial_iA_j]+[\partial_jZ,A_i]+
[[Z,A_j],A_i].
\end{align*}
Subtracting the similar expression for $(\Phi_j\Phi_i)(Z)$, we obtain
\begin{align*}
[\Phi_i,\Phi_j](Z)&=[Z,\partial_iA_j]-[Z,\partial_jA_i]+[[Z,A_j],A_i]-[[Z,A_i],A_j]=\\
&=
[Z,\partial_iA_j-\partial_jA_i-[A_i,A_j]],
\end{align*}
which vanishes by~\eqref{eq:Ai2} for $1\Le i,j\Le d-1$.

Let $y_i\in V$, $1\Le i\Le d-1$, denote the right-hand side of~\eqref{eq:4}:
$$
y_i:=\partial_dA_i-\partial_iB_d+[A_i,B_d].
$$
By construction, finding $C$ that satisfies~\eqref{eq:main} and~\eqref{eq:4} is equivalent to finding $y\in V$ that satisfies
\begin{equation}\label{eq:y}
\Phi_i(y)=y_i
\end{equation}
for all $i$, $1\Le i\Le d-1$. This system of non-homogenous linear differential equations is consistent if and only if one has
$$
\Phi_i(y_j)=\Phi_j(y_i)
$$
for all $i,j$, $1\Le i,j\Le d-1$. The latter is again implied by~\eqref{eq:Ai2}. Indeed, we have:
$$
\Phi_i(y_j)=\partial_i\partial_dA_j-\partial_i\partial_jB_d+
[\partial_iA_j,B_d]+[A_j,\partial_iB_d]+[\partial_dA_j,A_i]-[\partial_jB_d,A_i]+[[A_j,B_d],A_i].
$$
Subtracting the similar expression for $\Phi_j(y_i)$, we obtain
\begin{align*}
\Phi_i(y_j)-\Phi_j(y_i)&=\partial_i\partial_dA_j-\partial_j\partial_dA_i+
[\partial_iA_j,B_d]-[\partial_jA_i,B_d]+
[\partial_dA_j,A_i]-[\partial_dA_i,A_j]+[[A_j,B_d],A_i]-[[A_i,B_d],A_j]=\\
&=\partial_d(\partial_iA_j-\partial_jA_i-[A_i,A_j])+[\partial_iA_j-\partial_jA_i-[A_i,A_j],B_d],
\end{align*}
which vanishes by~\eqref{eq:Ai2} for $1\Le i,j\Le d-1$.

A consistent system of non-homogenous linear differential equations is equivalent to a consistent system of homogenous linear differential equations (doubled in size).
Therefore, by the hypothesis of the theorem, there exists $y\in V$ satisfying~\eqref{eq:y}, which implies the existence of $C\in \Mn_n(K)$ satisfying~\eqref{eq:main} and~\eqref{eq:4}. Thus, the matrices
$$
(A_1,\ldots,A_{d-1}, B_d+C)
$$
satisfy both~\eqref{eq:Ai1} and~\eqref{eq:Ai2}.
\end{proof}

\begin{example}\label{ex:replaceBis}
We will see that, in general, the $A_i$'s as in Theorem~\ref{theor-expl} have to be different from the original $B_i$'s. Assume that there is an element $t\in k$ such that $\partial_1(t)=1$ and $\partial_i(t)=0$ for all $i$, $2\Le i\Le d$ (e.g., $k$ is the field of rational functions $k=\Q(t_1,\ldots,t_d)$, $\partial_i=\partial_{t_i}$, and $t=t_1$). Let $d \Ge 2$ and suppose that $(n\times n)$-matrices
$A, A_1, \ldots A_d$ with entries in $K$ satisfy all the integrability conditions~\eqref{eq:Ai1} and~\eqref{eq:Ai2}.
Define
$$
B_1 := A_1,\ \ldots,B_{d-1} := A_{d-1},\
B_d := A_d + \diag(t),
$$
where $\diag(t)$ denotes the diagonal $(n\times n)$-matrix with $t\in k$ on the diagonal.
Then the new set of matrices $A, B_1,\ldots, B_d$ will still satisfy~\eqref{eq:Bi} but will not satisfy the integrability condition for the pair of derivations $\partial_{1}$ and $\partial_{d}$. Indeed, $\partial_1B_d-\partial_dB_1$ is the identity matrix, while $[B_1,B_d]$ vanishes.
\end{example}

\subsection{Relation to parameterized differential Galois groups}
Let $\U$ be a $D_k$-closure of $k$, $K_{\U}:=\Frac(\U\otimes_k K)$.
By~\cite[Proposition~8.11]{GGO}, we have
$$
\U\otimes_k\DMod(K,D_{K/k})\cong \DMod(K_{\U},D_{K_{\U}/\U})
$$
Below, we extend~\cite[Proposition~3.9(1)]{PhyllisMichael} to the case when $(k,D_k)$ is not necessarily differentially closed, which we also prove categorically.

\begin{theorem}\label{theor-main}
In the above notation, let $L$ be a PPV extension for $N_{K_{\U}}$ (which exists by~\cite[Theorem~3.5(1)]{PhyllisMichael}), $V:=N_L^{D_{K/k}}$, and  $\Gal^{D_K}(L/{K_{\U}})\subset \GL(V)$ be the parameterized differential Galois group of $L$ over ${K_{\U}}$. Suppose that $(k,D_k)$ is filtered-linearly closed. Then $N$ is isomonodromic if and only if $\Gal^{D_K}(L/{K_{\U}})$ is conjugate to a constant subgroup in $\GL(V)$.
\end{theorem}
\begin{proof}
Let $\Cat$ be the subcategory in $\DMod({K_{\U}},D_{{K_{\U}}/\U})$ that is $D_k$-tensor generated by $N_{K_{\U}}$ (see~\cite[Definition~4.19]{GGO}).
Recall that $L$ defines a $D_k$-fiber functor
$$
\omega:\Cat\to\Vect(\U)
$$
such that $\omega(N_{K_{\U}})=V$ and $\Gal^{D_K}(L/{K_{\U}})$ is the associated $D_k$-group (see~\cite[Theorem~5.5]{GGO}). More precisely, there is an equivalence of $D_k$-categories
$$
\Cat\cong\Rep\left(\Gal^{D_K}(L/{K_{\U}})\right)
$$
sending $N_{K_{\U}}$ to $V$. Thus, combining Propositions~\ref{prop-2},~\ref{prop-descent}, and~\ref{prop-const1}, we obtain the required result.
\end{proof}


\subsection{Analytic interpretation}\label{sec:anal}

We will now explain in more detail the relation between the analytic notion of isomonodromicity and Definition~\ref{defin-isomon}.
Let $f:X\to S$ be a holomorphic submersion between connected complex analytic manifolds with connected fibers such that $f$ is topologically locally trivial. Let $E$ be a holomorphic vector bundle on $X$ and $\nabla_{X/S}$ be a relative flat holomorphic connection on $E$ over $S$ (that is, the connection $\nabla_{X/S}$ is defined only along vector fields on $X$ that are tangent to the fibers of $f$). For a subset $\Sigma\subset S$, put $$X_{\Sigma}:=f^{-1}(\Sigma).$$ In particular, $X_s$ denotes the fiber of $f$ at a point $s\in S$.

Let $U$ be a sufficiently small open neighborhood of a point $s\in S$ such that there is a smooth isomorphism $$\phi:U\times X_s\stackrel{\sim}\longrightarrow X_U$$ whose restriction to $\{s\}\times X_s$ coincides with the embedding $X_s\hookrightarrow X_U$. This gives a collection of smooth isomorphisms $$\phi_{st}:X_s\stackrel{\sim}\longrightarrow X_t,$$ where $t\in U$, and a section $\sigma:U\to X_U$. Also, choose a trivialization $$\psi:\C^n\times U\stackrel{\sim}\longrightarrow \sigma^*E.$$ Then the connection $\nabla_{X/S}$ defines a family of relative monodromy representations $\rho_t$, $t\in U$, as the composition
$$
\begin{CD}
\pi_1\left(X_s,\sigma(s)\right)@>\sim>>\pi_1\left(X_t,\sigma(t)\right)
@>>>\GL_n\left((\sigma^*E)_t\right)@>\sim>>\GL_n\left((\sigma^*E)_s\right)@>\sim>>\GL_n(\C).
\end{CD}
$$
The isomorphism classes of the representations $\rho_t$ do not depend on the choices of~$\phi$ and $\psi$.

We say that $(E,\nabla_{X/S})$ is {\it analytically isomonodromic} if the isomorphism classes of the relative monodromy representations $\rho_s$ are locally constant over $S$ (for example, see~\cite[\S1]{Sabbah}). It is shown in~\cite[Proof of 1.2(1), first step]{Sabbah} that $(E,\nabla_{X/S})$ is analytically isomonodromic if and only if, for any point $s\in S$, there is an open neighborhood $s\in U\subset S$ such that $\nabla_{X/S}$ extends to a flat holomorphic connection on $E$ over $X_U$ (see also~\cite[Theorem~A.5.2.3]{Sibuya} for the case of one-dimensional fibers). This is a version of Definition~\ref{defin-isomon} in the analytic context.

Let us give an analytic interpretation of Theorem~\ref{thm:54}. By definition, $(E,\nabla_{X/S})$ is analytically isomonodromic along a holomorphic vector field $v$ on $S$ if and only if the relative monodromy representations are locally constant along (local) holomorphic curves on $S$ that are tangent to $v$. Thus, $(E,\nabla_{X/S})$ is analytically isomonodromic if and only if it is analytically isomonodromic along $d$ transversal vector fields on $S$, where $d:=\dim(S)$. Combining this with the property of analytic isomonodromicity discussed above, we obtain an analytic proof of following weaker version of Theorem~\ref{thm:54}:

Let $k$ (respectively, $D_k$) be the field of meromorphic functions (respectively, the space of meromorphic vector fields) on $S$. Analogously, define $K$ and $D_K$ for $X$ in place of $S$. Assume that a finite-dimensional $D_{K/k}$-module $N$ satisfies the partial isomonodromicity condition from Theorem~\ref{thm:54}. Then there is a point $s\in S$ such that $N$ is isomonodromic over the parameterized differential field $K_s$ over $k_s$, where $k_s$ is the field of meromorphic functions on open neighborhoods of $s$ in $S$ and $K_s$ is the field of meromorphic functions on open subsets in $X$ whose intersection with $X_s$ is dense in $X_s$.

In general, one cannot replace $K_s$ by the field of meromorphic functions along all $X_s$. However, the results from~\cite{JM,JMU} allow to similarly treat the latter case when the fibers of $f$ are complex projective lines with finite sets of points removed. Finally, the need of replacing $k$ by $k_s$ reflects the requirement for $(k,D_k)$ to be filtered-linearly closed in Theorem~\ref{thm:54}.

\subsection{Examples}\label{subsection-exampisom}

First, we provide a non-trivial example to Theorem~\ref{thm:54} showing that its statement is not true for an arbitrary field $(k,D_k)$. Namely, in the notation Example~\ref{example-new},  we construct a parameterized field $K$ over the field $(k,D_k)$ and a PPV extension $K\subset L$ such that $\Gal^{D_K}(L/K)\cong G$ and the solution space corresponds to the representation~$V$. We are very grateful to M.\,Singer, who suggested a general method for constructing PPV extensions with a given parameterized differential Galois group to us.

\begin{example}\label{examp-isomon}
The following example is based on iterated integrals. Let $k:=\Q(t_1,t_2)$ and $D_k:=k\cdot\partial_{t_1}\oplus k\cdot\partial_{t_2}$. Let
$$
F:=k{\big({\partial_x^{i}\partial_{t_1}^{j_1}\partial_{t_2}^{j_2} I_m}\big)},\quad i,j_1,j_2 \Ge 0,\, m=1,2,
$$
be the field of $\{\partial_x,\partial_{t_1},\partial_{t_2}\}$-rational functions in the differential indeterminates $I_1$ and $I_2$ over $k$. Put
$$
D_F:=F\cdot\partial_x\oplus F\cdot\partial_{t_1}\oplus F\cdot\partial_{t_2},
$$
and do the analogous for the other fields that appear in what follows. Then $(F,D_F)$ is a parameterized differential field over $(k,D_k)$. Let $L$ be a PPV extension of $F$ for the equation
\begin{equation}\label{eq:I1I2eq}
\partial_x(y)=\partial_{x}I_1\cdot I_2
\end{equation}
and let $I\in L$ be a solution of this equation (for example, see Proposition~\ref{prop-PPV} for the existence of $L$). A calculation shows that there are no elements $a\in F$ and linear $D_k$-operators $D$ with coefficients in $k$ such that
$
\partial_x(a)=D(\partial_{x}I_1\cdot I_2)$.
By Proposition~\ref{prop-intGalois}, we see that
$$
\Gal^{D_F}(L/F)\cong\G_a
$$
and, therefore, the elements $\partial^iI\in L$, $i\Ge 0$, are algebraically independent over $F$. Let
$
K\subset L
$
be the $\{\partial_x,\partial_{t_1},\partial_{t_2}\}$-subfield generated by
$$
\partial_x I_m,\,\partial_{t_1}I_m,\,\partial_{t_2}I_m,\quad m=1,2,\quad\quad
J_1:=\partial_{t_1}I-\partial_{t_1}I_1\cdot I_2-I_2/t_1,\quad
J_2:=\partial_{t_2}I-\partial_{t_2}I_1\cdot I_2+I_1/t_2.
$$
Since $I$ satisfies~\eqref{eq:I1I2eq} and $J_1,J_2\in K$, for all $(i,j_1,j_2)\ne (0,0,0)$, we have
$$\partial_x^{i}\partial_{t_1}^{j_1}\partial_{t_2}^{j_2}(I)\in K(I_2).$$
Therefore,
$$
L=K(I_1,I_2,I).
$$
One can show that $I_1,I_2,I$ are algebraically independent over $K$ using a characteristic set argument with respect to any orderly ranking of the derivatives with $I > I_1 > I_2$ \cite[Sections~I.8--10]{Kol}.
Put
$$
f_i:=\partial_x I_i\in K,\quad i=1,2,
$$
and consider the equation
\begin{equation}\label{eq:1955}
\partial_x(y)=A_{\partial_x}\cdot y,\quad y:={^t(y_1,y_2,y_3)},\quad
A_{\partial_x}:=\begin{pmatrix}
0&f_1&0\\
0&0&f_2\\
0&0&0
\end{pmatrix}.
\end{equation}
Then
$$
\Phi:=\begin{pmatrix}
1&I_1&I\\
0&1&I_2\\
0&0&1
\end{pmatrix}
$$
is the fundamental matrix for the equation~\eqref{eq:1955}, that is, $I$ is the iterated integral $\int_x{\left(f_1\cdot\int_x f_2\right)}$. Hence, $L$ is a PPV extension of $K$ for the equation~\eqref{eq:1955}.

In what follows, $U$ and $G$ are as in Example~\ref{example-new}. We see that $\Gal^{D_K}(L/K)$ is a linear $D_k$-subgroup in $U$, where $U$ acts on $\Phi$ by multiplication on the right. Explicitly, we have
$$
g(u_1,u_2,v)(I_i)=I_i+u_i,\quad g(u_1,u_2,v)(I)=I+I_1u_2+v.
$$
A calculation shows that $K\subset L^G$. By a dimension argument, we conclude that $\Gal^{D_K}(L/K)=G$. By Example~\ref{example-new}, the equation~\eqref{eq:1955} is not isomonodromic. On the other hand, this equation is $\partial_{t_i}$-isomonodromic, $i=1,2$ with the corresponding matrices given by
$$
B_i:=\Phi\cdot \tilde{B}_i\cdot \Phi^{-1} + \partial_{t_i}\Phi\cdot\Phi^{-1},\quad i=1,2,\quad\quad
\tilde{B}_1:=\begin{pmatrix}
0&1/t_1&0\\
0&0&0\\
0&0&0
\end{pmatrix},\quad
\tilde{B}_2:=\begin{pmatrix}
0&0&0\\
0&0&1/t_2\\
0&0&0
\end{pmatrix}.
$$
More explicitly,
$$
B_1:=\begin{pmatrix}
0&1/t_1+ \partial_{t_1}I_1& J_1\\
0&0& \partial_{t_1}I_2\\
0&0&0
\end{pmatrix},\quad B_2=\begin{pmatrix}
0& \partial_{t_2} I_1& J_2\\
0&0& 1/t_2+\partial_{t_2}I_2\\
0&0&0\end{pmatrix}
$$
Thus, we see that Theorem~\ref{thm:54} is not true for an arbitrary field $(k,D_k)$.
\end{example}

The purpose of the rest of the section is to show that, in Theorem~\ref{theor-main}, one really needs to take the extension of scalars from $k$ to $\U$ in order to obtain conjugacy to a constant group. Namely, we construct examples of an isomonodromic $D_{K/k}$-module $N$ such that there are PPV extensions of $K$ for $N$, but, for any PPV extension $L$, the parameterized differential Galois group $\Gal^{D_K}(L/K)$ is not a constant group and, thus, $\Gal^{D_K}(L/K)$ is not conjugate to a constant subgroup in $\GL_n(k)$.

The idea of the examples is as follows. We construct a parameterized
differential field $(K,D_K)$ over $(k,D_k)$ and $\omega\in
\Omega_{K/k}$ with $\dd\omega=0$ such that the $D_k$-submodule
$$M\subset H^1(K/k)$$ generated by $[\omega]$ is finite-dimensional
over $k$ and is not trivial as a $D_k$-module. By
Propositions~\ref{prop-intGalois} and~\ref{prop-PPV}, there are PPV
extensions of $K$ for the equation $\dd y=\omega$ and any of them
has the parameterized differential Galois group isomorphic to
$\left(M^\vee\right)^{D_k}$.
Since $\Rep{\left(\left(M^\vee\right)^{D_k}\right)}$ is $D_k$-equivalent to
a $D_k$-subcategory in $\DMod\left(K,D_{K/k}\right)$, the (faithful)
$D_k$-representation $V_{M^\vee}$ of $\left(M^\vee\right)^{D_k}$ (see
Remark~\ref{remark-nonconstGa}\eqref{en:719}) corresponds to an
isomonodromic $D_{K/k}$-module $N$ (see Proposition~\ref{prop-2}).

Let us give an explicit description. Suppose that $$D_k=k\cdot\partial_t,\quad
D_K=K\cdot\partial_x\oplus K\cdot\partial_t,\quad \text{and}\quad
[\partial_x,\partial_t]=0.$$ Let $\omega_x\in \Omega_{K/k}$ be such
that $\omega_x(\partial_x)=1$. Then $\omega=b\cdot\omega_x$ with
$b\in K$. Suppose that there exists a non-zero monic linear $\partial_t$-operator $D$ as in Example~\ref{examp-paramGG}\eqref{remark-expliintGalois}. Explicitly,
$$
D=\partial_t^n-\sum_{i=0}^{n-1} c_i\partial_t^i,\quad c_i\in k,
$$
is of the smallest order such that there is $a\in K$ with $D(b)=\partial_x(a)$ (see~Example~\ref{examp-GM}\eqref{en:770}). One can show that the differential module $N$ defined above corresponds to the following system of linear differential equations:
$$
\partial_x(y)=A_{\partial_x}\cdot y,\quad y:={^t(y_0,\ldots,y_{n})},\quad
A_{\partial_x}:=\begin{pmatrix}
0&0&\ldots&0\\
b&0&\ldots&0\\
\partial_t(b)&0&\ldots&0\\
\vdots&&\ldots&&\\
\partial^{n-1}_t(b)&0&\ldots&0
\end{pmatrix}.
$$
By Proposition~\ref{prop-PPV}, there is a PPV extension $L$ of $K$ for $N$ and $L=K(z,\partial_t(z),\ldots)$, where $\partial_x(z)=b$. By Proposition~\ref{prop-intGalois} and its proof combined with Example~\ref{examp-paramGG}\eqref{remark-expliintGalois}, the morphism of linear $D_k$-groups
$$
\Gal^{D_K}(L/K)\to\G_a,\quad g\mapsto g(z)-z
$$
induces an isomorphism
$$
\Gal^{D_K}(L/K)\stackrel{\sim}\longrightarrow \{u \in \G_a\:|\:Du=0\}\subset\G_a.
$$
The $\partial_x$-module $N$ is isomonodromic with
$$
A_{\partial_t}:=\begin{pmatrix}
0&0&0&\ldots&0\\
0&0&1&\ldots&0\\
\vdots&\vdots&&\ddots&\\
0&0&0&\ldots&1\\
a&c_0&c_1&\ldots&c_{n-1}
\end{pmatrix}.
$$
Now we give concrete
examples with $k=\Q(t)$ and $K$ being a generated by functions in $t$ and $x$. We construct $b\in K$ such that there exists a linear $\partial_t$-operator $D$ as above and the equation $Du=0$ in $u$ is non-trivial over $\Q(t)$.

\begin{example}\label{examp-1}
This examples comes from the algebraic independence of the derivatives of the incomplete Gamma-function (see~\cite{JRR}). Put
$$
E:=\Q{\left(t,x,\log x,x^{t-1}e^{-x}\right)},\quad D_E:=E\cdot\partial_x\oplus E\cdot\partial_t.
$$
By Proposition~\ref{prop-PPV}, there is a PPV extension $L$ of $E$
for the equation
\begin{equation}\label{eq:1393}
\partial_x(y)=x^{t-1}e^{-x}.
\end{equation}
As noted in~\cite[Example~7.2]{PhyllisMichael}, by~\cite{JRR}, there is an isomorphism
$
\Gal^{D_E}(L/E)\cong \G_a$.
Let $\gamma\in K$ be a solution of~\eqref{eq:1393} and put
$$
K:=E{\left(\partial_t(\gamma)-\gamma,\partial^2_t(\gamma)-\partial_t(\gamma),\ldots\right)}\subset L.
$$
Since $\Gal^{D_E}(L/E)\cong \G_a$, the parameterized Galois theory implies that $\gamma\notin K$. The element
$
b:=x^{t-1}e^{-x}\in K
$
satisfies
$$
D(b)=\partial_x(a),\quad D:=\partial_t-1,\quad a:=\partial_t(\gamma)-\gamma\in K.
$$
The operator $D$ is of the smallest order, because $b\notin \partial_x(K)$ as $\gamma\notin K$.
Note that $K$ is of infinite transcendence degree over $\Q(t,x)$, because $\Gal^{D_E}(K/E)\cong\G_a$.
\end{example}

\begin{example}\label{examp-2}
This example comes from the Gauss--Manin connection for the Legendre family of elliptic curves. Namely, put $K:=\Q(t,x,z)$, where $z^2=x(x-1)(x-t)$.
Then the element $$b:=1/z\in K$$
satisfies
$$
D(b)=\partial_x(a),\quad D:=-2t(t-1)\partial_t^2-(4t-2)\partial_t-1/2,\quad a:=z/(x-t)^2\in K.
$$
The operator $D$ is of the smallest order (for example, this follows from a monodromy argument, see~\cite[\S2.10]{Clemens}).
\end{example}

\bibliographystyle{model1b-num-names}
\bibliography{difftanncat}

\end{document}